\documentclass[11pt]{amsart}
\usepackage{amsfonts}
\usepackage{amsmath}
\usepackage{amscd}
\usepackage{amsthm}
\usepackage{enumerate}

\usepackage{graphicx}

\usepackage{color}

\usepackage{hyperref}

\usepackage[margin=1in]{geometry}

\newtheorem{theorem}{Theorem} 
\newtheorem*{theorem*}{Theorem}
\newtheorem{lemma}[theorem]{Lemma}
\newtheorem{definition}[theorem]{Definition}
\newtheorem{proposition}[theorem]{Proposition}
\newtheorem{conjecture}[theorem]{Conjecture}
\newtheorem{corollary}[theorem]{Corollary}

\theoremstyle{remark}
\newtheorem{rmk}[theorem]{Remark}

\newcommand{\pp}{\mathbb{P}}

\newcommand{\rr}{\mathbb{R}}

\newcommand{\nn}{\mathbb{N}}

\newcommand{\eq}{\begin{equation}}
\newcommand{\en}{\end{equation}}

\newcommand{\ev}{\mathbb{E}}

\newcommand{\eps}{\varepsilon}

\newcommand{\X}{\mathcal X}
\newcommand{\Y}{\mathcal Y}

\newcommand{\R}{\mathcal R}

\newcommand{\WW}{\overline{{\mathcal W}^N}}

\newcommand{\GG}{\overline{\mathcal{G}^N}}

\newcommand{\F}{\mathbb F}

\numberwithin{equation}{section} \numberwithin{theorem}{section}

\title{Interacting particle systems at the edge of multilevel Dyson Brownian motions}

\author{Vadim Gorin}
\address{Department of Mathematics, Massachusetts Institute of Technology, MA, USA and Institute for Information Transmission Problems of Russian Academy of Sciences,  Russia}
\email{vadicgor@gmail.com}
\author{Mykhaylo Shkolnikov}
\address{Department of Mathematics, Princeton University, Princeton, NJ, USA}
\email{mshkolni@gmail.com}

\begin{document}

\begin{abstract}
We study the joint asymptotic behavior of spacings between particles at the edge of multilevel
Dyson Brownian motions, when the number of levels tends to infinity. Despite the global
interactions between particles in multilevel Dyson Brownian motions, we observe a decoupling
phenomenon in the limit: the global interactions become negligible and only the local interactions
remain. The resulting limiting objects are interacting particle systems which can be described as
Brownian versions of certain totally asymmetric exclusion processes. 
 This is the first appearance of a particle system with local interactions in the context of
general $\beta$ random matrix models.
\end{abstract}

\maketitle


\section{Introduction}

The main theme of this article is a connection between general $\beta$ random matrix models and
interacting particle systems with local interactions. When $\beta=2$, that is in the case of
Hermitian random matrices, a first rigorous connection of this kind was established fifteen years
ago in \cite{J-TASEP}. In that paper Johansson proved that the large time fluctuations for the
current of the totally asymmetric simple exclusion process (TASEP) started from the step initial
condition are governed by the Tracy--Widom distribution. This distribution had appeared previously
as the limit of the fluctuations of the largest eigenvalue of random Hermitian matrices from the
Gaussian Unitary Ensemble (GUE) when one lets the size of the matrix tend to infinity.

\medskip

A corresponding connection between GUE matrices of \emph{finite} size and TASEP--like processes
has been established later by means of the combinatorial RSK correspondence in \cite{Bar},
\cite{GTW}, \cite{OC1} and \cite{OC2}, by a stochastic coupling procedure for Markov chains due to
Diaconis and Fill \cite{DF} in  \cite{BF}, \cite{GS-TASEP}, \cite{Ferrari}, and by other methods
in \cite{Warren}, \cite{Nordenstam}. The results of the above articles lead to statements of the
following flavor. One starts from a discrete space stochastic dynamics on interlacing arrays with
$N$ levels and $N(N+1)/2$ particles (see Figure \ref{Figure_array} for an example of such an array
with $N=3$). The restriction of that dynamics to the $N$ rightmost (or leftmost) particles on the
$N$ levels turns out to be an interacting particle system with local interactions (typically a
version of the TASEP). The latter converges in the diffusive scaling limit to a Brownian particle
system with local interactions (which is therefore typically referred to as Brownian TASEP). On
the other hand, the diffusive scaling limit of the $N$ particles on the top level is typically
given by the Dyson Brownian Motion, which is a natural evolution of the $N$ eigenvalues of an
$N\times N$ GUE matrix.

\medskip

All of the above results are restricted to the case of $\beta=2$. When $\beta=1$, that is in the
case of real symmetric random matrices, an \emph{asymptotic} connection to TASEP of a similar type
as above is known (see \cite{S-TASEP}, \cite{BFPS}, \cite{FSW} and also \cite{PrahSpohn},
\cite{BR}), but this case is much less understood conceptually. In addition, while many of the
studied particle systems have far reaching generalizations (see \cite{BigMac}, \cite{BG_lect},
\cite{BP-RSK}, \cite{BP_lect} and \cite{GS}), the interactions between particles are non-local for
the range of parameters corresponding to general $\beta$ random matrix models. Naively, one might
conclude that there are no connections between general $\beta$ random matrix models and
interacting particle systems with local interactions. In the present article we do find such a
connection, proving this naive conclusion to be spurious.

\begin{figure}[h]
\begin{center}
 {\scalebox{0.67}{\includegraphics{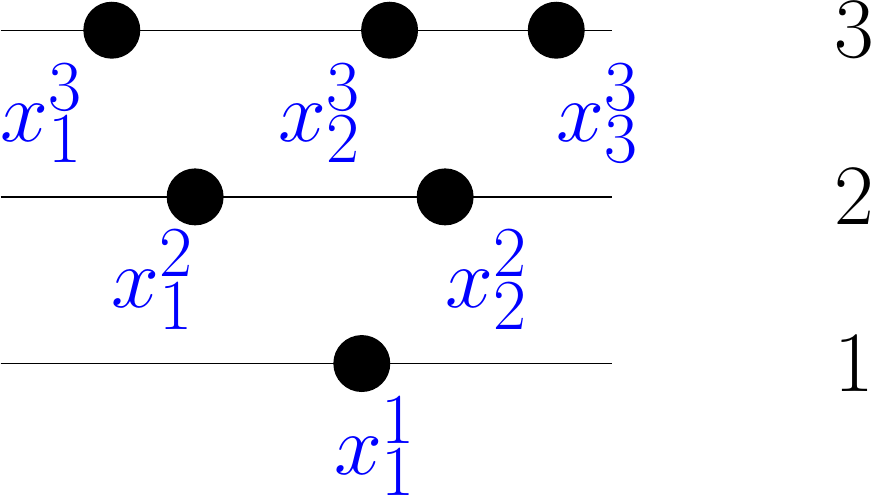}}}
\end{center}
\caption{Array of interlacing particles with $N=3$ levels and $N(N+1)/2=6$ particles.}
\label{Figure_array}
\end{figure}

A key object in our results is the $N(N+1)/2$--dimensional diffusion process $X(t;N)$, $t\ge0$
introduced in \cite{GS}. The state space of that process is the Gelfand--Tsetlin cone $\GG$
defined via \eq\label{GTcone} \GG=\left\{x=(x^k_i)_{1\leq i\leq k\leq
N}\in\rr^{N(N+1)/2}:\;x^{k+1}_i\leq x^k_i\leq x^{k+1}_{i+1},\,1\le i\le k\le N-1\right\}. \en We
refer to the $x^k_i\,$'s as positions of particles and visualize them as in Figure
\ref{Figure_array}. Although the process $X(t;N)$ can be defined for any $\beta>0$, we restrict
ourselves to the case $\beta\ge 4$ throughout the paper due to the technical difficulties arising
for $0<\beta<4$. The case $\beta\ge4$  is distinguished by the property that the particles (almost
surely) never collide with each other at all times $t>0$, and then the process $X(t;N)$ solves the
systems of SDEs
\begin{equation} \label{eq_betaWarren_Intro} \mathrm{d}X^k_i(t;N)=\sum_{j=1}^{k-1} \frac{(\beta/2-1)\,
\mathrm{d}t}{X^k_i(t;N)-X^{k-1}_j(t;N)} -\sum_{j\neq i} \frac{(\beta/2-1)\,
\mathrm{d}t}{X^k_i(t;N)-X^k_j(t;N)}+\mathrm{d}W^k_i(t),\quad 1\leq i\leq k\leq N
\end{equation}
where $W^k_i$, $1\leq i\leq k\leq N$ are independent standard Brownian motions. Here we choose the
initial condition $X(0;N)$ to be the zero vector. The definition of $X(t;N)$ in \cite{GS} was
motivated, in particular, by the following remarkable properties: The evolution of the
$N$-dimensional vector $\big(X_1^N(t;N),\dots,X_N^N(t;N)\big)$ is given by the celebrated Dyson
Brownian motion (see e.g. \cite{Mehta}, \cite{AGZ} and \cite{For}). On the other hand, the
distribution of $X(t;N)$ at a fixed time $t$ is given by the (appropriately scaled)
$\beta$-Hermite corners process, which is an extension to general $\beta>0$ of the process of
eigenvalues of all $k\times k$ top left corners ($k$ running from $1$ to $N$) of an $N\times N$
GUE Hermitian random matrix. In particular, the distribution of
$\big(X_1^N(t;N),\dots,X_N^N(t;N)\big)$ at a fixed time $t$ has a probability density proportional
to
\begin{equation}
\label{eq_beta_Hermite} \prod_{1\le i<j \le N} (x_j-x_i)^\beta \; \prod_{i=1}^N
\exp\left(-\frac{x_i^2}{2t}\right),
\end{equation}
We refer to Section \ref{Section_setup} and to \cite{GS} for more details on the process $X(t;N)$.

\medskip

We will be concerned with the asymptotic behavior of the particles at the \emph{edge} of the
process $X(t;N)$, that is with the behavior of the rightmost particles on the different levels in
$X(t;N)$ when $N$ becomes large. For $\beta=2$ the result of \cite{TW-U}, \cite{F1} yields the
convergence in distribution
\begin{equation}
\label{eq_TW} \frac{X^N_N(N;N)-2N}{N^{1/3}} \, \stackrel{N\to\infty}{\longrightarrow}\, F_2,
\end{equation}
where $F_2$ is know known as the GUE Tracy--Widom distribution. It is convenient for us to choose
$t=N$ in \eqref{eq_TW}, but we note that due to the Brownian scaling property of $X(t;N)$ we can
choose any other fixed time $t$ instead and absorb the change into the normalization terms.

\medskip

The convergence in \eqref{eq_TW} can be extended to similar statements for several coordinates
$X^N_N,\,X^N_{N-1},\,\ldots,\,X^N_{N-j}$ with $j$ being kept finite as $N$ tends to infinity, and
also to \emph{dynamic} statements describing the joint distribution of these coordinates at
several times. The latter results identify the edge scaling limit of the $\beta=2$ Dyson Brownian
Motion with the Airy line ensemble; dynamic scaling limits are also available in the cases
$\beta=1,4$, see \cite{KNT}, \cite{CH}, \cite{Sodin}, \cite{OT} and
 references therein. For general real values of $\beta$
only the convergence of the fixed time distributions is known, but the results extend to very
general random matrix distributions, see \cite{RRV}, \cite{BEY}, \cite{KRV}, \cite{Shch},
\cite{BFG}.

\medskip

In a similar yet different direction, one can study \emph{multilevel} edge scaling limits at a
\textit{fixed time}, that is the asymptotics of the joint distributions of the coordinates
$X^{N-i}_{N-j}(t;N)$ with varying $i,j$, but fixed $t$. One usually takes the indices $i$ an $j$
to be on the order of $N^{2/3}$. Interestingly, for $\beta=2$ the resulting joint distribution
converges in the limit $N\to\infty$ (after proper centering and scaling) to the same Airy line
ensemble (see \cite{FN} and \cite{Sodin}). A similar phenomenon has been also demonstrated for
$\beta=1$ in \cite{Sodin}.

We note that in such results the information about the spacings between the extremal particles on
adjacent levels (that is the differences $X^{N+1-i}_{N+1-i}(N;N)-X_{N-i}^{N-i}(N;N)$,
$i=1,2,\ldots$) is lost. This has to do with the fact that the scaling one needs to apply to these
spacings to see a non-trivial limiting behavior is not $N^{1/3}$ of \eqref{eq_TW}. Our first
result is a limit theorem for such spacings at a \textit{fixed time}.

\begin{theorem}[Theorem \ref{Theorem_convergence_fixed_time}] \label{Theorem_convergence_fixed_time_intro}
For every fixed $k\in\nn$ the $k$-dimensional random vectors
\begin{multline}
\label{eq_vector_diffs}
 \bigg( X_N^N\Big( \frac{2N}{\beta};N\Big)-X_{N-1}^{N-1}\Big(\frac{2N}{\beta};N\Big),\,
 X_{N-1}^{N-1}\Big(\frac{2N}{\beta};N\Big)-X_{N-2}^{N-2}\Big(\frac{2N}{\beta};N\Big),\,\ldots,\\
 X_{N-k+1}^{N-k+1}\Big(\frac{2N}{\beta};N\Big)-X_{N-k}^{N-k}\Big(\frac{2N}{\beta};N\Big) \bigg)
\end{multline}
converge in distribution in the limit $N\to\infty$ to a random vector whose components are
independent identically distributed, each according to the Gamma distribution with density
\begin{equation} \label{eq_Gammas}
\frac{1}{\Gamma(\beta/2)} \Big(\frac{\beta}{2}\Big)^{\beta/2}\,x^{\beta/2-1}\,e^{-\frac{\beta}{2}x}\,.
\end{equation}
\end{theorem}

\begin{rmk} The absence of an additional space scaling in \eqref{eq_vector_diffs}
can be explained by the following heuristics: the $(N-1)$ differences $X_{N+1-i}^{N+1-i}\big(\frac{2N}{\beta};N\big)-X_{N-i}^{N-i}\big(\frac{2N}{\beta};N\big)$, $i=1,2,\ldots,N-1$ are positive and sum up to approximately $2N$ (in view of \eqref{eq_TW} and its general $\beta$ analogue). This suggests that the typical size of each such difference is of constant order.
\end{rmk}

\medskip

We remark that while $X(t;N)$ does no longer satisfy the SDEs \eqref{eq_betaWarren_Intro} for
small values of $\beta$, an analogue of Theorem \ref{Theorem_convergence_fixed_time_intro} remains
true for all $\beta\ge 1$, see Section \ref{Section_fixed}. In particular, for $\beta=1,2,4$ we
get the following statements about random matrices from Gaussian Orthogonal, Unitary and
Symplectic ensembles (GOE, GUE, GSE, respectively), which we (surprisingly) were not able to find
in the literature.

\begin{corollary} \label{Corollary_GOE} Consider a GOE random matrix of size $N\times N$,
normalized such that the variance of its diagonal elements is equal to $2N$. For $k=1,2,\ldots,N$
write $\lambda^{(k)}(N)$ for the largest eigenvalue of the $(N+1-k)\times(N+1-k)$ top-left
submatrix of that matrix. Then, for any fixed $K\in\nn$ the random vectors
$$
\big(\lambda^{(1)}(N)-\lambda^{(2)}(N),\,\lambda^{(2)}(N)-\lambda^{(3)}(N),\,\ldots,\,\lambda^{(K)}(N)-\lambda^{(K+1)}(N)\big)
$$
converge in distribution as $N\to\infty$ to a random vector with i.i.d.\ Gamma-distributed entries
\eqref{eq_Gammas} with $\beta=1$.
\end{corollary}
\begin{corollary} \label{Corollary_GUE} Consider a GUE random matrix of size $N\times N$,
normalized such that the variance of its diagonal elements is equal to $N$. For $k=1,2,\ldots,N$
write $\lambda^{(k)}(N)$ for the largest eigenvalue of the $(N+1-k)\times(N+1-k)$ top-left
submatrix of that matrix. Then, for any fixed $K\in\nn$ the random vectors
$$
\big(\lambda^{(1)}(N)-\lambda^{(2)}(N),\,\lambda^{(2)}(N)-\lambda^{(3)}(N),\,\ldots,\,\lambda^{(K)}(N)-\lambda^{(K+1)}(N)\big)
$$
converge in distribution as $N\to\infty$ to a random vector with i.i.d.\ mean $1$ exponentially
distributed entries.
\end{corollary}
\begin{corollary} \label{Corollary_GSE} Consider a GSE random matrix of size $N\times N$,
normalized such that the variance of its diagonal elements is equal to $N/2$. For $k=1,2,\ldots,N$
write $\lambda^{(k)}(N)$ for the largest eigenvalue of the $(N+1-k)\times(N+1-k)$ top-left
submatrix of that matrix. Then, for any fixed $K\in\nn$ the random vectors
$$
\big(\lambda^{(1)}(N)-\lambda^{(2)}(N),\,\lambda^{(2)}(N)-\lambda^{(3)}(N),\,\ldots,\,\lambda^{(K)}(N)-\lambda^{(K+1)}(N)\big)
$$
converge in distribution as $N\to\infty$ to a random vector with i.i.d.\ Gamma-distributed entries
\eqref{eq_Gammas} with $\beta=4$.
\end{corollary}

Our next aim is to study the \emph{dynamic multilevel} edge scaling limits. We first note that for
the $N^{1/3}$ scaling as in \eqref{eq_TW} no such results are available in the literature. In a
similar, yet different multilevel dynamics coming from Gaussian random matrices \emph{multilevel}
edge scaling limits at \textit{multiple times} were obtained recently in \cite{Sodin}, however we
believe that the corresponding \textit{dynamic} limits should be different from those in our
setting. We refer to \cite{BF} for a related discussion in the case of $\beta=2$.

\medskip

Instead of looking at the extremal eigenvalues on levels of distance of order $N^{2/3}$ from each
other as in \cite{Sodin}, we continue our study of the spacings between the rightmost particles on
adjacent levels, and show the existence of a \textit{dynamic} scaling limit of those in our next
theorem. Hereby, for $n=1,2,\ldots$ we use the notation $\mathcal{C}^n$ for the space of
continuous functions from $[0,\infty)$ to $[0,\infty)^n$, endowed with the topology of
componentwise uniform convergence on compact sets.

\begin{theorem}[Theorem \ref{Theorem_multitime_convergence}] \label{Theorem_multitime_convergence_intro}
For any $\beta\ge 4$ and $k=1,2,\ldots$ the distribution of the process
\begin{multline}
\label{eq_differences_vector_intro}
 \bigg( X_N^N\Big( \frac{2N}{\beta}+t;N\Big)-X_{N-1}^{N-1}\Big(\frac{2N}{\beta}+t;N\Big),\,
 X_{N-1}^{N-1}\Big( \frac{2N}{\beta}+t;N\Big)-X_{N-2}^{N-2}\Big(\frac{2N}{\beta}+t;N\Big),\,\ldots,\\
 X_{N-k+1}^{N-k+1}\Big(\frac{2N}{\beta}+t;N\Big)-X_{N-k}^{N-k}\Big(\frac{2N}{\beta}+t;N\Big) \bigg),\quad t\ge 0
\end{multline}
on $\mathcal{C}^k$ converges weakly to that of a process
\eq\label{whatisR}
 \big( R_1(t), R_2(t),\ldots,R_k(t) \big), \quad t\ge 0.
\en
Moreover, the process of \eqref{whatisR} is a stationary Markov process.
\end{theorem}

Note that for every fixed $N$ each particle in the dynamics \eqref{eq_betaWarren_Intro} was
interacting with all other particles on its own level and on the level below it. In particular,
there was no reason to expect that the joint dynamics of the rightmost particles on the different
levels form a Markov process. Nonetheless, a Markov process arises once one passes to the
scaling limit.

\smallskip

Next, we identify the dynamics of the process of \eqref{whatisR}.

\begin{theorem}[Theorem \ref{Theorem_multitime_convergence}] \label{Theorem_multitime_convergence_intro_2}
For any $\beta\ge 4$ and $k=1,2,\ldots$ consider the $(k+1)$-dimensional process
$\big(Z_1^{(k)},\,Z_2^{(k)},\,\ldots,\,Z_{k+1}^{(k)}\big)$ given by the weak solution of the system
of SDEs
\begin{equation}
\label{eq_SDE_for_coords}
\begin{split}
& \mathrm{d}Z_i^{(k)}(t)=\frac{(\beta/2-1)\,\mathrm{d}t}{Z_i^{(k)}(t)-Z_{i+1}^{(k)}(t)}+\mathrm{d}B_i(t),\quad i=1,2,\ldots,k,\\
& \mathrm{d}Z_{k+1}^{(k)}(t)=\mathrm{d}t+\mathrm{d}B_{k+1}(t)
\end{split}
\end{equation}
where $B_1,\,B_2,\,\ldots,\,B_{k+1}$ are i.i.d.\ one-dimensional standard Brownian motions and the
initial condition is chosen such that $Z_{k+1}^{(k)}(0)=0$ and
$Z_i^{(k+1)}(0)-Z_{i+1}^{(k+1)}(0)$, $i=1,2,\ldots,k$ are i.i.d.\ Gamma distributed with density
\eqref{eq_Gammas}. Then the following equality in law between processes holds:
$$
\big(R_1,\,R_2,\,\ldots,\,R_k \big)\stackrel{d}{=}
\big(Z_1^{(k)}-Z_2^{(k)},\,Z^{(k)}_2-Z^{(k)}_3,\,\ldots,\, Z_k^{(k)}-Z_{k+1}^{(k)}\big).
$$
\end{theorem}

\begin{rmk} By passing to the differences in \eqref{eq_SDE_for_coords} one can easily produce a closed system of SDEs for the process $\big(R_1,\,R_2,\,\ldots,\,R_k \big)$. This system is given in Theorem \ref{Theorem_multitime_convergence} and we will
work with it for the most part.
\end{rmk}

\begin{rmk}
Due to the singularities in the drift coefficients, the existence and uniqueness of a weak
solution to the system of SDEs \eqref{eq_SDE_for_coords} and its version for the differences
\eqref{eq_SDE_for_diffs} do not follow directly from classical existence and uniqueness theorems
for SDEs, and we devote Subsection \ref{Section_uniqueness} to the resolution of this question.
\end{rmk}

\begin{rmk} \label{Remark_curious_property}
Theorem \ref{Theorem_multitime_convergence_intro} implies a curious property of the solution to the
system of SDEs \eqref{eq_SDE_for_coords} started with stationary Gamma-distributed differences: for
any $1\le \ell\le m\le k$ the distribution of the process of $(m-\ell+1)$ differences
$$
\big(Z_\ell^{(k)}-Z_{\ell+1}^{(k)},\,Z^{(k)}_{\ell+1}-Z^{(k)}_{\ell+2},\,\ldots,\,Z_m^{(k)}-Z_{m+1}^{(k)}\big)
$$
depends only on $(m-\ell)$ and, in particular, does not depend on $k$. In Section
\ref{Section_Limit_properties} we discuss an alternative way to approach this property.
\end{rmk}

\begin{rmk}
Informally, one would like to say that the \emph{infinite-dimensional} process
$\big(R_1,\,R_2,\,\ldots\big)$ is given by the differences between consecutive coordinates in the
infinite system of SDEs
\begin{equation} \label{eq_SDE_for_coords_inf}
 \mathrm{d}Z_i(t)=\frac{(\beta/2-1)\,\mathrm{d}t}{Z_i(t)-Z_{i+1}(t)}+\mathrm{d}B_i(t),\quad i=1,2,\ldots
\end{equation}
started from an initial condition where the differences are i.i.d.\ Gamma distributed with the
density of \eqref{eq_Gammas}. We make this interpretation rigorous by building such a solution of
\eqref{eq_SDE_for_coords_inf} (and, hence, also the corresponding process of differences) by
relying on a \textit{consistent} sequence of solutions to \eqref{eq_SDE_for_coords} with growing
values of $k$.
\end{rmk}

We want to emphasize that, in sharp contrast to \eqref{eq_betaWarren_Intro}, the interaction terms
in the systems of SDEs \eqref{eq_SDE_for_coords} and \eqref{eq_SDE_for_coords_inf} are of a
\textit{local form}, in the sense that each coordinate interacts only with its immediate neighbor.
As explained in \cite[Section 5.2]{OO} the solutions of \eqref{eq_SDE_for_coords} and
\eqref{eq_SDE_for_coords_inf} should be thought of as Brownian versions of suitable totally
asymmetric exclusion processes.

\medskip

Although Theorem \ref{Theorem_multitime_convergence_intro} is restricted to $\beta\ge 4$, it is
enlightening to consider the case $\beta=2$ as well. For $\beta=2$ the process $X(t;N)$, $t\ge0$ is
given by the Warren process of \emph{interlacing reflecting} Brownian motions introduced in
\cite{Warren} (for an explanation we refer to \cite{GS-TASEP}, \cite{GS} where we have shown that
both the solutions to \eqref{eq_betaWarren_Intro} and the Warren process arise in the scaling limit
of certain  discrete interacting particle systems whose jumps rates depend analytically on
$\beta>0$). The case of $\beta=2$, that is of the Warren process, is special in that its
restriction to the rightmost particles $\big(X^1_1(t;N),X^2_2(t;N),\ldots,X^N_N(t;N)\big)$, $t\ge0$
is already a Markov process for all finite $N$. This process is the Brownian version of the
classical TASEP which was introduced by Glynn and Whitt in \cite{GW}. It is defined inductively:
$X_1^1(t;N)$, $t\ge0$ is a standard Brownian motion, whereas for $i>1$, the process $X_i^i(t;N)$,
$t\ge0$ is an independent standard Brownian motion \emph{reflected} on the trajectory of the
process $X_{i-1}^{i-1}(t;N)$, $t\ge0$, so that $X_i^i(t;N)\ge X_{i-1}^{i-1}(t;N)$ for all $t$. As
one would expect, this process can be obtained as a diffusive scaling limit of the classical TASEP,
see \cite{GW}.

\medskip

For the classical TASEP itself analogues of our Theorems
\ref{Theorem_convergence_fixed_time_intro}, \ref{Theorem_multitime_convergence_intro} and
\ref{Theorem_multitime_convergence_intro_2} are known. Indeed, for large times the classical TASEP
converges locally to its stationary translation invariant version on $\mathbb Z$, see \cite{Rost}.
The stationary distribution depends on a parameter $p$: at each lattice point of $\mathbb Z$ there
is a particle with probability $p$, independently of all other lattice points. This means that the
spacings between consecutive particles are i.i.d.\ geometrically distributed which is precisely
the discrete space analogue of the i.i.d.\ exponentially distributed spacings in the $\beta=2$
version of Theorem \ref{Theorem_convergence_fixed_time_intro}.

\smallskip

Further, the evolution of a
single \emph{tagged} particle in the stationary TASEP is itself a Markov chain and the jump rates are functions of the parameter $p$ 
(see \cite[Example 3.2]{Spitzer} and \cite[Chapter VII, Corollary 4.9]{Liggett}, and in addition
\cite[Introduction]{Kipnis} for the connection with Burke's Theorem on queues arranged in series).
This fact implies that the evolution of any number $k$ of spacings between adjacent particles is a
Markov chain of birth-and-death type. Our Theorems \ref{Theorem_multitime_convergence_intro},
\ref{Theorem_multitime_convergence_intro_2} contain a continuous general $\beta$ version of this
property.

\bigskip

The rest of the article is organized as follows. In Section \ref{Section_setup} we give the
definitions related to the process $X(t;N)$, $t\ge0$. Section \ref{Section_fixed} is devoted to the
study of the asymptotic behavior of the fixed time distributions of that process and to the proof
of Theorem \ref{Theorem_convergence_fixed_time_intro}. In this section we rely on various
previously known results from random matrix theory, such as the Wigner semi-circle law for
$\beta$-Hermite ensembles and large deviations estimates both in the bulk and at the edge of the
spectrum. In Section \ref{Section_convergence} we prove Theorems
\ref{Theorem_multitime_convergence_intro} and \ref{Theorem_multitime_convergence_intro_2}. The
convergence is proved via martingale problem techniques in the spirit of Stroock and Varadhan. In
the same section we also prove weak uniqueness for the systems of SDEs \eqref{eq_SDE_for_coords}
and \eqref{eq_SDE_for_diffs}. Our argument is based on a (local in time) Girsanov change of measure
which locally reduces our interacting particle system to a system of non-interacting Bessel
processes. Finally, in Section \ref{Section_Limit_properties} we outline how techniques from the
theory of semigroups and linear evolution equations can be used to provide independent proofs for
the properties of the solutions to \eqref{eq_SDE_for_coords} and \eqref{eq_SDE_for_diffs} discussed
in Remark \ref{Remark_curious_property} above.

\medskip

\noindent {\bf Acknowledgement.} We would like to thank Alexei Borodin, Paul Bourgade, Amir Dembo, Alice Guionnet, Sasha Sodin and Ofer Zeitouni for many fruitful discussions. In particular, we thank Alice Guionnet and Ofer Zeitouni for showing us the integration by parts trick employed in the proof of Lemma \ref{Lemma_inverse_expectation}. V.~G.\ was partially supported by the NSF grant DMS-1407562.

\section{Preliminaries}

\label{Section_setup}

We start by recalling the definition of the Gelfand--Tseitlin cone $\GG$ in \eqref{GTcone} and by introducing the Hermite $\beta$ corners process in the following definition.

\begin{definition}\label{def_betacorner} The Hermite $\beta>0$ corners process of variance $t>0$ is the probability distribution on $\GG$ whose density (with respect to the Lebesgue measure) is proportional to
\begin{equation}
\label{eq_beta_Hermite_corners}
 \prod_{i<j} (x_j^N-x_i^N)
 \prod_{i=1}^N \exp\left(- \frac{(x_i^N)^2}{2t}\right)
 \prod_{k=1}^{N-1} \prod_{1\le i<j\le k} (x_j^k-x_i^k)^{2-\beta} \prod_{a=1}^k \prod_{b=1}^{k+1} |x^k_a-x^{k+1}_b|^{\beta/2-1}.
\end{equation}
\end{definition}

\smallskip

For $\beta=1$, $2$ and $4$ the Hermite $\beta$ corners process arises as the joint distribution of
the eigenvalues of a Hermitian Gaussian random matrix and its corners with real (GOE), complex
(GUE) and quaternion entries (GSE), respectively. We refer to \cite{Neretin} for a proof and to
\cite[Introduction]{GS} for a more detailed discussion.

\medskip

In \cite{GS} we had introduced for every fixed $\beta>0$ and $N=1,2,\ldots$ a stochastic process
$X(t;N)$, $t\ge0$ taking values in $\GG$ (this process was called $Y^{mu}(t)$, $t\ge0$ in
\cite[Introduction]{GS} and the $\theta$ there equals to $\beta/2$ here). We have shown in
\cite{GS} that, for any $\beta>0$ and $t>0$ the distribution of $X(t;N)$ is given by the Hermite
$\beta$ corners process of variance $t$. Throughout this article we focus on the case $\beta \ge
4$ in which the following theorem can be taken as the definition of $X(t;N)$, $t\ge0$.

\begin{theorem}[\cite{GS}] \label{Theorem_GS_multilevel} For every $\beta \ge 4$ and $N=1,2,\dots$ the system of SDEs
\begin{equation}
\label{eq_betaWarren} \mathrm{d}X^k_i(t;N)=\sum_{j=1}^{k-1} \frac{(\beta/2-1)\,
\mathrm{d}t}{X^k_i(t;N)-X^{k-1}_j(t;N)} -\sum_{j\neq i} \frac{(\beta/2-1)\,
\mathrm{d}t}{X^k_i(t;N)-X^k_j(t;N)}+\mathrm{d}W^k_i(t),\quad 1\leq i\leq k\leq N,
\end{equation}
 has a unique weak solution taking values in $\GG$, such that for each $t\ge 0$ the distribution of $X(t;N)$ is given by the Hermite $\beta$ corners process of variance $t$.
In particular, it has the initial condition $X(0;N)=0$. Here $W^k_i$, $1\leq i\leq k\leq N$ are
i.i.d.\ one-dimensional standard Brownian motions.
\end{theorem}

\smallskip

This theorem immediately implies that, for any $1\le k\le N$ the restriction of the process
$X(t;N)$, $t\ge0$ to the first $k$ levels, that is to the coordinates $x_i^j$, $1\le i \le j \le
k$, has the same law as the process $X(t;k)$, $t\ge0$. A more delicate restriction property is a
part of the following proposition.

\begin{proposition}[\cite{GS}] \label{Proposition_GS_restriction}
For every fixed $1\le k\le N$ the restriction of the process $X(t;N)$, $t\ge0$ to the $k$-th level, that is to the coordinates $x_i^j$, $1\le i \le k$, has the law of a $\beta$ Dyson Brownian motion. In other words, it admits the semimartingale decomposition
$$
\mathrm{d}X^k_i(t;N)=\sum_{j\neq i} \frac{(\beta/2)\,\mathrm{d}t}{X^k_i(t;N)-X^k_j(t;N)}+\mathrm{d}B_i(t),\quad 1\leq i\leq k
$$
in its own filtration, with $B_i$, $1\le i \le k$ being i.i.d.\ one-dimensional standard Brownian
motions. In particular, for any $t>0$ the distribution of the random vector
$\big(X^k_1(t),\,X^k_2(t),\,\ldots,\,X^k_k(t)\big)$ on \eq\label{wedge}
\overline{\mathcal{W}^k}:=\big\{x\in\rr^k:\;x_1\le x_2\le\cdots\le x_k\big\} \en has a density
with respect to the Lebesgue measure proportional to
\begin{equation}
\label{eq_time_t_DBM} \prod_{1\le i<j \le k} (x^k_j-x^k_i)^\beta \prod_{i=1}^k
\exp\left(-\frac{(x_i^k)^2}{2 t} \right).
\end{equation}
\end{proposition}

\section{Analysis of the fixed time distribution}

\label{Section_fixed}

In this section we prove several asymptotic properties, as $N\to\infty$, for the distribution of
the random vector $X(t;N)$ at a fixed time $t>0$, or in other words for the Hermite $\beta$ corners
process of Definition \ref{def_betacorner}. The main results of this section are the following four
(related) statements. Note that while the dynamic results of Theorem
\ref{Theorem_multitime_convergence} below are restricted to the case $\beta\ge 4$, we allow for any
$\beta\ge 1$ throughout this section. It is plausible that certain statements of this section
remain true for $0<\beta<1$ as well, but we do not address this case here. Throughout this section
we write $C$ and $c$ for positive constants whose values might change from line to line.

\begin{lemma} \label{lemma_sum_inverse}
For any fixed $\beta\ge 1$, $T_0>0$ and $t\ge0$:
\begin{equation} \label{eq_x13}
\lim_{N\to\infty} \; \ev\bigg[\,\bigg|\,\sum_{i=1}^{N-1}\frac{1}{X_N^N(NT_0+t;N)-X_{i}^N(NT_0+t;N)}- \sqrt{\frac{2}{\beta T_0}} \bigg|\,\bigg]=0.
\end{equation}
Moreover, the convergence is uniform on compact sets in $T_0$ and $t$.
\end{lemma}

\begin{rmk}
The $N\to\infty$ limits of the linear statistics similar to the sum in Lemma \ref{lemma_sum_inverse} are typically given by integrals with respect to the Wigner semi-circle law, see e.g. \cite{Mehta}, \cite{AGZ} and \cite{For}. An appropriate integral with respect the semi-circle law gives the correct answer in our case as well, however we need additional arguments due to the singularity of the summands in \eqref{eq_x13}.
\end{rmk}

\begin{lemma} \label{lemma_one_inverse}
For any fixed $\beta\ge 1$, $T_0>0$ and $t\ge0$:
$$
\lim_{N\to\infty} \;\ev \bigg[\frac{1}{X_N^N(NT_0+t;N)-X_{N-1}^N(NT_0+t;N)} \bigg]=0.
$$
Moreover, the convergence is uniform on compact sets in $T_0$ and $t$.
\end{lemma}

\begin{rmk}
In fact, $N^{-1/3}\big(X_N^N(NT_0+t;N)-X_{N-1}^N(NT_0+t;N)\big)$ converges in law to a limiting
random variable (see \cite{RRV}) which explains why Lemma \ref{lemma_one_inverse} should hold.
However, the singularity of the function $x\mapsto\frac{1}{x}$ does not allow to apply the result
of \cite{RRV} here directly and we need additional arguments. Estimates similar to Lemma
\ref{lemma_one_inverse} can be found in \cite{BEY}, but the exact statement we need is also not
available there.
\end{rmk}

\begin{lemma} \label{Lemma_sum_inverse_cross_level}
For any fixed $\beta\ge 1$, $T_0>0$ and $t\ge0$:
\begin{equation}
\lim_{N\to\infty} \; \ev\bigg[\,\bigg|\,\sum_{i=1}^{N-2}\frac{1}{X_N^N(NT_0+t;N)-X_{i}^{N-1}(NT_0+t;N)}- \sqrt{\frac{2}{\beta T_0}} \bigg|\,\bigg]=0.
\end{equation}
Moreover, the convergence is uniform on compact sets in $T_0$ and $t$.
\end{lemma}

\begin{theorem}\label{Theorem_convergence_fixed_time}
For any fixed $\beta\ge 1$, $T_0>0$, $t\ge0$ and $k=1,\,2,\,\ldots$ the $k$-dimensional random vectors
\begin{multline*}
\Big( X_N^N(NT_0+t;N)-X_{N-1}^{N-1}(NT_0+t;N),\,
 X_{N-1}^{N-1}(NT_0+t;N)-X_{N-2}^{N-2}(NT_0+t;N),\,\ldots,\\
 X_{N-k+1}^{N-k+1}(NT_0+t;N)-X_{N-k}^{N-k}(NT_0+t;N) \Big)
\end{multline*}
converge in distribution in the limit $N\to\infty$ to a random vector with i.i.d.\ Gamma
distributed components of probability density
$$
\frac{1}{\Gamma(\beta/2)}\bigg(\frac{\beta}{2T_0}\bigg)^{\beta/4}\,x^{\beta/2-1}\,e^{-\sqrt{\frac{\beta}{2T_0}}x}.
$$
\end{theorem}

The rest of this section is devoted to the proofs of Lemmas \ref{lemma_sum_inverse}, \ref{lemma_one_inverse} and \ref{Lemma_sum_inverse_cross_level}, and to that of Theorem \ref{Theorem_convergence_fixed_time}. We start by noting that the Hermite $\beta$ corners process is invariant under a diffusive rescaling of time and space, which can be seen immediately from \eqref{eq_beta_Hermite_corners}. Therefore, we have the following equality in distribution:
\eq\label{eq_in_dist}
 X(NT_0+t;N)\stackrel{d}{=}X\left(\frac{2N}{\beta};N\right)\cdot\bigg(\frac{ \beta T_0+\beta\frac{t}{N}}{2}\bigg)^{1/2}.
\en Consequently, it suffices to consider the case $T_0=2/\beta$, $t=0$ in Lemmas
\ref{lemma_sum_inverse}, \ref{lemma_one_inverse}, and \ref{Lemma_sum_inverse_cross_level}, and in
Theorem \ref{Theorem_convergence_fixed_time}. With this choice, the constant $\sqrt{\frac{2}{\beta
T_0}}$ in Lemma \ref{lemma_sum_inverse} is equal to $1$, and the constant
$\sqrt{\frac{\beta}{2T_0}}$ in Theorem \ref{Theorem_convergence_fixed_time} equals $\beta/2$. To
simplify the notation we also define
$$
\X_i:=X^N_i\left(\frac{2N}{\beta};N\right), \quad i=1,\,2,\,\ldots,\,N.
$$
Note that the distribution of the random vector $\X:=(\X_1,\,\X_2,\,\ldots\,,\X_N)$ is given by
\eqref{eq_time_t_DBM} with $t=2N/\beta$ and $k=N$.

\medskip

Next, we recall several known asymptotic properties of $\X$. Firstly, for every $\eps>0$ the
coordinates $\X_1,\,X_2,\,\ldots,\,\X_N$ belong to the interval $[-(2+\eps)\,N, (2+\eps)\,N]$ with
high probability. More precisely, \cite[Theorem 1 and Equation (1.4)]{LR} show that for every
$\eps>0$ there exists a constant $C=C(\eps)>0$ such that for each $N=1,2,\ldots$
\begin{equation} \label{eq_LR_estimates}
\pp\big( |\X_i|>(2+\eps)\,N\text{ for some }i=1,\,2,\,\ldots,\,N \big)\le
\begin{cases} C\exp(-N \eps^{3/2}/C),& \text{if }\eps<1,\\ C\exp(-N \eps^2/C), &\text{if } \eps
\ge 1.
\end{cases}
\end{equation}

\smallskip

Further, the Wigner semi-circle law for beta-ensembles (see \cite{J-CLT}, \cite[Section 6.2]{Dum},
\cite[Section 2.6]{AGZ})
 implies that for any continuous bounded function $f$ on $\rr$ we have
\begin{equation}
\label{eq_Wigner_semicircle}
 \lim_{N\to\infty}\;
 \frac{1}{N}\,\sum_{i=1}^N f\left(\frac{\X_i}{N}\right)= \frac{1}{2\pi}\int_{-2}^2 f(s)\,(4-s^2)^{1/2}\,\mathrm{d}s
\end{equation}
in the sense of convergence in probability. In addition, we need a certain large deviations
estimate around the semi-circle law. For our purposes the simplest way to proceed is to use the
bulk rigidity estimates established in \cite{BEY_bulk}, although probably the same conclusion
(more specifically, the estimate \eqref{eq_deviation_bound} below) can be also drawn from other
large deviations estimates in the literature, such as \cite[Section 2.6]{AGZ}. For
$i=1,\,2,\,\ldots,\,N$ we define $\gamma_i$ through the identity
$$
\frac{1}{2\pi} \int_{-\infty}^{\gamma_i/N} (4-s^2)^{1/2}\,\mathrm{d}s = \frac{i}{N}.
$$
Then, \cite[Theorem 3.1]{BEY_bulk} asserts that for every $\eps,\kappa>0$ there exist
$c=c(\eps,\kappa)>0$, $C=C(\eps,\kappa)>0$ (both independent of $N$) such that for all
$i\in\{\lfloor \eps N\rfloor, \lfloor \eps N\rfloor,\dots,\lfloor (1-\eps)N\rfloor\}$ we have
\begin{equation}
\label{eq_bulk_rigidity} \pp\big( |\X_i-\gamma_i|>N^{\kappa} \big)< C\,\exp\left(-N^c/C\right).
\end{equation}

\smallskip

Our proof of Lemma \ref{lemma_sum_inverse} relies on the following auxiliary lemma.

\begin{lemma}\label{Lemma_inverse_expectation} For any $\beta\ge 1$ we have
\eq
\lim_{N\to\infty} \; \ev\,\bigg[\,\sum_{i=1}^{N-1} \frac{1}{\X_N-\X_i}\bigg]=1.
\en
\end{lemma}

\begin{proof}
Let $Z_N$ denote the normalization constant of the $\beta$ Hermite ensemble, that is the integral of the expression \eqref{eq_time_t_DBM} over the set \eqref{wedge} with $t=2N/\beta$ and $k=N$. Integration by parts yields
\begin{eqnarray*}
\ev\bigg[\,\sum_{i=1}^{N-1} \frac{1}{\X_N-\X_i}\bigg]
&=& \frac{1}{Z_N}\,\int_{\WW} \sum_{i=1}^{N-1} \frac{1}{x_N-x_i} \,
\prod_{i<j} (x_j-x_i)^{\beta}\,\prod_{i=1}^N \exp\left(-\frac{\beta x_i^2}{4N}\right)\,\mathrm{d}x_i\\
&=&\frac{1}{\beta Z_N}\int_{\WW} \frac{\partial}{\partial x_N}
\bigg( \prod_{i<j} (x_j-x_i)^{\beta}\bigg)\,
\prod_{i=1}^N \exp\left(-\frac{\beta x_i^2}{4N}\right)\,\mathrm{d}x_i \\
&=&-\frac{1}{\beta Z_N}\int_{\WW}  \prod_{i<j} (x_j-x_i)^{\beta}\,
\frac{\partial}{\partial x_N}\bigg( \prod_{i=1}^N \exp\left(-\frac{\beta x_i^2}{4N}\right)\bigg)
\,\prod_{i=1}^N \,\mathrm{d}x_i \\
&=& \frac{1}{2N Z_N}\int_{\WW}  x_N\,\prod_{i<j} (x_j-x_i)^{\beta}
\,\prod_{i=1}^N \exp\left(-\frac{\beta x_i^2}{4N}\right)\,\mathrm{d}x_i=\frac{\ev[\X_N]}{2N}.
\end{eqnarray*}
It remains to use the well-known fact that
$$
\lim_{N\to\infty} \frac{\ev[\X_N]}{2N}=1,
$$
see e.g. \cite[Corollary 3 and discussion following it]{LR}.
\end{proof}

An argument similar to the one in the proof of Lemma \ref{Lemma_inverse_expectation} yields the following lemma that will be useful for us further below.

\begin{lemma}\label{Lemma_inverse_expectation_squared} For any $\beta\ge 2$ we have
$$
 \liminf_{N\to\infty} \; \ev\bigg[\,\bigg(\sum_{i=1}^{N-1} \frac{1}{\X_N-\X_i}\bigg)^2\,\bigg]\ge 1,\qquad
 \limsup_{N\to\infty} \; \ev\bigg[\,\bigg(\sum_{i=1}^{N-1} \frac{1}{\X_N-\X_i}\bigg)^2\,\bigg]\le \frac{\beta}{\beta-1}.
$$
\end{lemma}

\begin{rmk}
With additional efforts one probably can prove that the expectation in Lemma
\ref{Lemma_inverse_expectation_squared} converges to $1$, but we will not need this fact here.
\end{rmk}

\begin{proof}[Proof of Lemma \ref{Lemma_inverse_expectation_squared}]
Starting with the identity
$$
 \left(\frac{\partial}{\partial x_N}\right)^2\bigg( \prod_{1\le i<j\le N} (x_j-x_i)^{\beta}\bigg)
=\bigg( -\beta\sum_{i=1}^{N-1}\frac{1}{(x_N-x_i)^2}
+\beta^2\bigg(\sum_{i=1}^{N-1} \frac{1}{x_N-x_i}\bigg)^2\bigg)\,\prod_{1\le i<j\le N}
(x_j-x_i)^{\beta},
$$
integrating by parts as in the proof of Lemma \ref{Lemma_inverse_expectation} twice, and appealing to \cite[Corollary 3 and discussion following it]{LR} we obtain
$$
 \lim_{N\to\infty}\; \ev\bigg[ -\frac{1}{\beta}\,\sum_{i=1}^{N-1}\frac{1}{(\X_N-\X_i)^2}+
\bigg(\sum_{i=1}^{N-1} \frac{1}{\X_N-\X_i}\bigg)^2\bigg]=1.
$$
The lemma now follows from the elementary inequalities
$$
\frac{\beta-1}{\beta}\,\bigg(\sum_{i=1}^{N-1} \frac{1}{\X_N-\X_i} \bigg)^2
\le -\frac{1}{\beta}\,\sum_{i=1}^{N-1}\frac{1}{(\X_N-\X_i)^2}
+ \bigg(\sum_{i=1}^{N-1} \frac{1}{\X_N-\X_i}\bigg)^2
\le \left(\sum_{i=1}^{N-1} \frac{1}{\X_N-\X_i}\right)^2. \qedhere
$$
\end{proof}

\begin{proof}[Proof of Lemma \ref{lemma_sum_inverse}]
We need to show that
\eq\label{claim:L1}
 \lim_{N\to\infty} \;\ev \left[\left| \sum_{i=1}^{N-1} \frac{1}{\X_N-\X_i} -1 \right|\right],
\en
that is the $L^1$-convergence of the positive random variables $\sum_{i=1}^{N-1} \frac{1}{\X_N-\X_i}$ to the constant $1$. This amounts to establishing the convergence of expectations and the convergence in probability. In view of Lemma \ref{Lemma_inverse_expectation}, the claim \eqref{claim:L1} can be therefore reduced to the statements
\begin{eqnarray}
&& \forall\,\eps>0:\quad \lim_{N\to\infty} \;\pp\bigg(\sum_{i=1}^{N-1} \frac{1}{\X_N-\X_i}>1-\eps\bigg)=1, \label{eq_x8} \\
&& \forall\,\eps>0:\quad \lim_{N\to\infty} \;\pp\bigg(\sum_{i=1}^{N-1}
\frac{1}{\X_N-\X_i}<1+\eps\bigg)=1. \label{eq_x9}
\end{eqnarray}
We start with \eqref{eq_x8} and let $\delta>0$ be a constant to be chosen later. According to \cite[Theorem 1]{LR}, with probability tending to $1$ we have $\X_N\le(2+\delta)N$, and on that event
\begin{equation}\label{eq_x11}
\begin{split}
\sum_{i=1}^{N-1} \frac{1}{\X_N-\X_i}
\ge\frac{1}{N}\sum_{i=1}^{N-1} \frac{1}{2+2\delta-\X_i/N}
& =\frac{1}{N}\sum_{i=1}^{N} \frac{1}{2+2\delta-\X_i/N} - \frac{1}{N}\,\frac{1}{2+2\delta-\X_N/N} \\
& \ge \frac{1}{N}\sum_{i=1}^{N} \frac{1}{2+2\delta-\X_i/N} - \frac{1}{N\delta}.
\end{split}
\end{equation}
Moreover, on the event $\X_N\le(2+\delta)N$ one can write the first term in the latter lower bound in the form of the left-hand side of \eqref{eq_Wigner_semicircle} with a function $f$ which is continuous and bounded. Applying the semi-circle law we can therefore conclude
\begin{equation}
\label{eq_x10}
\lim_{N\to\infty} \; \frac{1}{N}\sum_{i=1}^{N} \frac{1}{2+2\delta-\X_i/N}=\frac{1}{2\pi} \int_{-2}^2 \,\frac{1}{2+\delta-s}
\,(4-s^2)^{1/2}\,\mathrm{d}s
\end{equation}
in probability. We further note that the right-hand side of \eqref{eq_x10} is a continuous function of $\delta\ge 0$ which
equals to $1$ when $\delta=0$ due to the simple computation
\begin{equation} \label{eq_integral_evalutation}
 \int_{-2}^2 \, \frac{1}{2-s}\,(4-s^2)^{1/2}\,\mathrm{d}s\,\stackrel{s=2\cos u}{=}\,\int_{0}^\pi \frac{2\sin u}{2-2\cos u}\,
2\sin u\,\mathrm{d}u = 2\,\int_0^\pi (1+\cos u)\,\mathrm{d}u=2\pi.
\end{equation}
We conclude that there exists a $\delta>0$ such that the two sides of \eqref{eq_x10} are greater than $1-\eps/3$. Fixing such a choice of $\delta$, considering $N>3/(\delta\eps)$ and combining \eqref{eq_x11} with \eqref{eq_x10} we obtain \eqref{eq_x8}.

\medskip

Now, we turn our attention to \eqref{eq_x9}. We argue by the contradiction and assume that there exists a $\xi>0$ such that there are arbitrary large $N$ with
$$
\pp\bigg(\sum_{i=1}^{N-1} \frac{1}{\X_N-\X_i}>1+\eps\bigg)\ge \xi.
$$
We let $\delta>0$ be a constant to be chosen later and for every fixed $N$ introduce the events
\begin{multline*}
 A^>:=\left\{\sum_{i=1}^{N-1} \frac{1}{\X_N-\X_i}>1+\eps\right\},\quad  A^\approx :=\left\{1-\delta\le \sum_{i=1}^{N-1}
 \frac{1}{\X_N-\X_i}\le 1+\eps\right\},\\ A^< :=\left\{\sum_{i=1}^{N-1}
 \frac{1}{\X_N-\X_i}<1-\delta\right\}.
\end{multline*}
With the notation $\mathbf{1}_A$ for the indicator function of an event $A$ we can now write
\begin{equation} \label{eq_x12}
\begin{split}
 \ev\bigg[\,\sum_{i=1}^{N-1} \frac{1}{\X_N-\X_i}\bigg] &=
 \ev\bigg[\,\mathbf{1}_{A^>}\,\sum_{i=1}^{N-1} \frac{1}{\X_N-\X_i}\bigg]
+\ev\bigg[\,\mathbf{1}_{A^\approx}\,\sum_{i=1}^{N-1} \frac{1}{\X_N-\X_i}\bigg]
+\ev\bigg[\,\mathbf{1}_{A^<}\,\sum_{i=1}^{N-1} \frac{1}{\X_N-\X_i}\bigg] \\
& \ge (1+\eps)\,\pp(A^>)+(1-\delta)\,\pp(A^\approx) = (1+\eps)\,\pp(A^>)+(1-\delta)(1-\pp(A^>)-\pp(A^<)) \\
& \qquad\qquad\qquad\qquad\qquad\qquad\qquad\;\; =1-\delta+(\delta+\eps)\,\xi-(1-\delta)\,\pp(A^<).
\end{split}
\end{equation}
For any fixed $\delta<\frac{\eps\,\xi}{1-\xi}$ the latter lower bound is strictly greater than $1$ and bounded away from $1$ for all $N$ large enough due to \eqref{eq_x8}. On the other hand, the first expression in \eqref{eq_x12} tends to $1$ in the limit $N\to\infty$. This is the desired contradiction.
\end{proof}

\begin{proof}[Proof of Lemma \ref{lemma_one_inverse}]
For each $0<\eps,\eta<1$ and $N\in\nn$ we introduce the events
\begin{eqnarray*}
&& A(\eps;N):=\{\X_{N-1}<(2+\eps)N\}, \\
&& B(\eps,\eta;N):=\bigg\{\frac{1}{N} \sum_{i=1}^{N-\lceil \eps N\rceil} \frac{1}{(2+2\eps)N-\X_i}>1-\eta,
\;\; \X_{N-1}-\X_{N-\lceil\eps N\rceil} > \eps^2 N\bigg\}.
\end{eqnarray*}
In addition, we define
\begin{equation}
\label{eq_f_density_def} f(x;y_1,\dots,y_{N-1}):=\begin{cases} \prod_{i=1}^{N-1}(x-y_i)^{\beta}
\exp\left(-\frac{\beta x^2}{4N}\right),& \text{if }x>y_{N-1}\\ 0,&\text{otherwise.} \end{cases}
\end{equation}
Note that $f(x;y_1,\dots,y_{N-1})$ is proportional (up to a constant independent of $x$) to the conditional probability density of $\X_N$ given $\X_1=y_1,\,\X_2=y_2,\,\ldots,\,\X_{N-1}=y_{N-1}$.

\medskip

First, we will give an upper bound on the conditional expectations of the form
$$
\ev\bigg[\frac{1}{\X_N-\X_{N-1}}\bigg|\,A(\eps;N)\cap B(\eps,\eta;N)\bigg].
$$
To this end, we pick a value $(y_1,\,y_2,\,\ldots,\,y_{N-1})$ of $\big(\X_1,\,\X_2,\,\ldots,\,\X_{N-1}\big)$ for which the event $A(\eps;N)\cap B(\eps,\eta;N)$ occurs, choose $\delta>0$ and $x$ such that $y_{N-1}<x<x+\delta<(2+2\eps)N$, and consider the ratio
\begin{equation}
\label{eq_dens_ratio}
\frac{f(x+\delta;y_1,\dots,y_{N-1})}{f(x;y_1,\dots,y_{N-1})}=\left(\frac{x-y_{N-1}+\delta}{x-y_{N-1}}\right)^\beta
\,\exp\left(-\frac{\beta\delta(2x+\delta)}{4 N}\right) \, \prod_{i=1}^{N-2} \left(1+\frac{\delta}{x-y_{i}}\right)^\beta.
\end{equation}
Thanks to $x+\delta<(2+2\eps) N$ we have
\begin{equation}
\label{eq_x3} \exp\left(-\frac{\beta\delta(2x+\delta)}{4N}\right)>\exp\bigl(-\beta\delta(1+\eps)\bigr).
\end{equation}
Further, using $x-y_i>\eps^2 N$ for $i\le N-N\eps$, and recalling the elementary inequality $1+u>\exp((1-\eps)u)$ valid for all small enough positive $u$ we obtain
\begin{equation} \label{eq_x1}
\prod_{i=1}^{N-2} \Big(1+\frac{\delta}{x-y_{i}}\Big)
> \exp\bigg(\frac{\beta\delta(1-\eps)}{N}\sum_{i=1}^{N-\lceil \eps N\rceil} \frac{1}{x/N-y_i/N}\bigg)
> \exp\big( \beta\delta(1-\eps)(1-\eta)\big)
\end{equation}
for all $N$ large enough. Moreover, a combination of \eqref{eq_x3} and \eqref{eq_x1} yields
\begin{equation}
\label{eq_x4}
\frac{f(x+\delta;y_1,\dots,y_{N-1})}{f(x;y_1,\dots,y_{N-1})}>\left(\frac{x-y_{N-1}+\delta}{x-y_{N-1}}\right)^\beta
\exp\left(-\beta\delta(\eta+2\eps)\right).
\end{equation}

\smallskip

The inequality \eqref{eq_x4} leads to an upper bound on the conditional expectation of $\frac{1}{\X_N-\X_{N-1}}$. Indeed, we can write
\begin{equation} \label{eq_x16}
\begin{split}
&\ev\bigg[\frac{1}{\X_N-\X_{N-1}}\bigg|\,A(\eps;N)\cap B(\eps,\eta;N)\bigg] \\
&\qquad\qquad = \ev\bigg[\frac{1}{\X_N-\X_{N-1}}\bigg|\,A(\eps;N)\cap B(\eps,\eta;N),\;\X_N-\X_{N-1}<(\eta+2\eps)^{-1/2}\bigg] \\
&\qquad\qquad\;\;\;\;\;\times \pp\Big(\X_N-\X_{N-1}<(\eta+2\eps)^{-1/2}\big|\,A(\eps;N)\cap B(\eps,\eta;N)\Big) \\
&\qquad\qquad\;\;\; + \ev\bigg[\frac{1}{\X_N-\X_{N-1}}\bigg|\,A(\eps;N)\cap B(\eps,\eta;N),\;\X_N-\X_{N-1}\ge (\eta+2\eps)^{-1/2}\bigg]\\
&\qquad\qquad\;\;\;\;\;\;\;\; \times\pp\Big(\X_N-\X_{N-1}\ge(\eta+2\eps)^{-1/2}\big|\,A(\eps;N)\cap B(\eps,\eta;N)\Big).
\end{split}
\end{equation}
Now, we can divide $\big[0,(\eta+2\eps)^{-1/2}\big)$ into disjoint intervals of length $\delta$
and use \eqref{eq_x4} to bound the possible increase of the conditional expectation of
$\frac{1}{\X_N-\X_{N-1}}$ on the event that $\X_N-\X_{N-1}$ belongs to one such interval, as we
move from the rightmost to the leftmost interval of this type. It follows that the conditional
expectation in the first summand on the right-hand side of \eqref{eq_x16} can be bounded above by
a constant $C<\infty$ independent of $N$, $\eps$ and $\eta$ as long as $\eps<1$ and $\eta<1$. To
bound the conditional probability in the first summand on the right--hand side of \eqref{eq_x16},
we use \eqref{eq_x4} again, which implies that the conditional density of
$\frac{1}{\X_N-\X_{N-1}}$ on $\big[0,(\eta+2\eps)^{-1/2}\big)$ is less than a (universal) multiple
of that same density on $\big[(\eta+2\eps)^{-1/2},(\eta+2\eps)^{-1}\big)$. This and the trivial
fact that the conditional probability of $\frac{1}{\X_N-\X_{N-1}}$ being in the latter interval is
at most one show that the first conditional probability in \eqref{eq_x16} is bounded above by
$C(\eta+2\eps)^{1/2}$, with $C<\infty$ being a universal constant as long as $\eps<1$ and
$\eta<1$.

Moreover, the second conditional expectation on the right-hand side of \eqref{eq_x16} is clearly
less than $(\eta+2\eps)^{1/2}$, while the conditional probability is less than $1$. All in all, we
see that for each $0<\eps,\eta<1$ and all $N$ large enough it holds
\begin{equation}
\label{eq_x15}
\ev\bigg[\,\frac{1}{\X_N-\X_{N-1}}\bigg|\,A(\eps;N)\cap B(\eps,\eta;N)\bigg]<C(\eta+2\eps)^{1/2}.
\end{equation}

\smallskip

Our next aim is to understand the behavior of $\frac{1}{\X_N-\X_{N-1}}$ on the complement of the event $A(\eps;N)\cap B(\eps,\eta;N)$. To this end, we note that if we drop the term $\prod_{i=1}^{N-2} (x-y_i)^\beta$ in \eqref{eq_f_density_def}, then the expectation of $\frac{1}{\X_N-\X_{N-1}}$ will increase or stay the same. One can see this by noting that the modified probability density stochastically dominates the original one thanks to the monotonicity of the quotient of the two. It follows that for \emph{all} $y_i$ we have the bound
\eq
 \ev \bigg[\frac{1}{\X_N-y_{N-1}}\bigg|\X_i=y_i,\;i=1,\,2,\,\ldots,\,N-1 \bigg]
\le \frac{\int_{y_{N-1}}^\infty (x-y_{N-1})^{\beta-1} \exp\left(-\frac{\beta x^2}{4N}\right)\,\mathrm{d}x}
 {\int_{y_{N-1}}^\infty (x-y_{N-1})^{\beta} \exp\left(-\frac{\beta x^2}{4N}\right)\,\mathrm{d}x}\,.
\en

\smallskip

We are interested in the ratio of the integrals
\eq \label{eq_x5}
\begin{split}
\int_{y_{N-1}}^\infty (x-y_{N-1})^\gamma \exp\left(-\frac{\beta x^2}{4N}\right)\,\mathrm{d}x &=
N^{\frac{\gamma+1}{2}}  \int_{y_{N-1}/\sqrt{N}}^\infty \Big(z-\frac{y_{N-1}}{\sqrt{N}}\Big)^\gamma
\exp\left(-\frac{\beta z^2}{4}\right)\,\mathrm{d}z \\
&=N^{\frac{\gamma+1}{2}}  \int_0^\infty
z^\gamma\,\exp\left(-\frac{\beta}{4}\Big(z+\frac{y_{N-1}}{\sqrt{N}}\Big)^2\right) \,\mathrm{d}z
\end{split}
\en
with $\gamma=\beta$ and $\gamma=\beta-1$. Clearly, as long as $y_{N-1}/\sqrt{N}$ is bounded in absolute value by a uniform constant, the ratio is bounded above by a uniform constant. On the other hand, when $y_{N-1}/\sqrt{N}$ is negative and of large absolute value, the dominant contribution to both integrals comes from $z$'s in the neighborhood of $-y_{N-1}/\sqrt{N}$ and in this case the ratio of the two integrals is less than one. Finally, when $y_{N-1}/\sqrt{N}$ is a large positive number (which is typical, since we expect $y_{N-1}\approx 2N$), the dominant contribution to both integrals comes from $z$'s in the neighborhood of zero. Therefore, in this regime the integrals of \eqref{eq_x5} have the asymptotics
\eq\label{eq_x6}
\begin{split}
&\; N^{\frac{\gamma+1}{2}}  \int_0^\infty z^\gamma\,\exp\left(-\frac{\beta}{4}\Big(z^2+\frac{y_{N-1}^2}{N}+\frac{2\,z\,y_{N-1}}{\sqrt{N}}\Big)\right)\,\mathrm{d}z\\
&\approx N^{\frac{\gamma+1}{2}} \exp\left(-\frac{\beta}{4}\,\frac{y_{N-1}^2}{N}\right)
\int_0^\infty z^{\gamma}\,\exp\left(-\frac{\beta}{2}\,\frac{z\,y_{N-1}}{\sqrt{N}}\right)\,\mathrm{d}z \\
&= \frac{N^{\frac{\gamma+1}{2}}}{y_{N-1}^{\gamma+1}}
\exp\left(-\frac{\beta}{4}\,\frac{y_{N-1}^2}{N} \right)
\int_0^\infty z^\gamma\,\exp\left(-\frac{\beta}{2}\, z \right)\,\mathrm{d}z\,.
\end{split}
\en Taking the ratio of the latter expression with $\gamma=\beta-1$ and that with $\gamma=\beta$
we find a universal constant $C<\infty$ such that \eq
\ev\left[\frac{1}{\X_N-y_{N-1}}\bigg|\X_i=y_i,\;i=1,\,2,\,\ldots,\,N-1 \right] \le
C\,\frac{y_{N-1}}{\sqrt N} \en when $\frac{y_{N-1}}{\sqrt N}\ge C$. All in all, we see that there
is a universal constant $C>0$ such that for all $y_1,\,y_2,\,\ldots,\,y_{N-1}$
\begin{equation}
\label{eq_x7}
\ev \left[\frac{1}{\X_N-y_{N-1}}\bigg|\X_i=y_i,\;i=1,\,2,\,\ldots,\,N-1 \right]
\le C\,\bigg(\frac{y_{N-1}}{\sqrt N}\mathbf{1}_{\{y_{N-1}>0\}}+1\bigg).
\end{equation}

\medskip

To conclude the proof we write
\eq\label{eq_x14}
\begin{split}
 \ev \left[\frac{1}{\X_N-\X_{N-1}}\right]
=\ev \left[\frac{1}{\X_N-\X_{N-1}}\bigg|\{\X_{N-1}\le 3N\}\cap A(\eps;N)\cap B(\eps,\eta;N)\right]
\,\pp\Big(A(\eps;N)\cap B(\eps,\eta;N)\Big) \\
+ \ev \left[\frac{1}{\X_N-\X_{N-1}}\bigg|\{\X_{N-1}\le 3N\}\backslash\Big(A(\eps;N)\cap B(\eps,\eta;N)\Big)\right]
 \pp \biggl(\{\X_{N-1}\le 3N\}\backslash \Big(A(\eps;N)\cap B(\eps,\eta;N)\Big)\biggr)\\
+ \ev \left[\frac{1}{\X_N-\X_{N-1}}\,\mathbf{1}_{\{\X_{N-1}>3N\}}\right].
\end{split}
\en
We analyze the three summands separately and start with the third summand. By conditioning on $\X_1,\,\X_2,\,\ldots,\,\X_{N-1}$ and using \eqref{eq_x7} it can be bounded above by
$$
\ev\bigg[C\,\bigg(\frac{\X_{N-1}}{\sqrt N}\mathbf{1}_{\{\X_{N-1}>0\}}+1\bigg)\,\mathbf{1}_{\{\X_{N-1}>3N\}}\bigg].
$$
Moreover, for $z>3\sqrt{N}$ the estimate \eqref{eq_LR_estimates} yields
$$
\pp(\X_{N-1}/\sqrt{N}>z)\le C \exp\left(-(z-2\sqrt{N})^2/C\right).
$$
Hence, the third summand is bounded above by
$$
C\,\int_{3\sqrt{N}}^\infty \exp\left(-(z-2\sqrt{N})^2/C\right)\,\mathrm{d}z+3\,C\sqrt{N}\exp\left(-N/C\right)+C\exp\left(-N/C\right)
$$
which rapidly decays to zero as $N\to\infty$.

\smallskip

To estimate the second summand we combine the inequality \eqref{eq_LR_estimates}, Wigner's semi-circle law \eqref{eq_Wigner_semicircle} and the large deviation estimate \eqref{eq_bulk_rigidity} to conclude that for each $\eps>0$ there exists an $\eta=\eta(\eps)>0$, so that $\lim_{\eps\to 0} \eta(\eps)=0$ and
\begin{equation}\label{eq_deviation_bound}
\pp\biggl(\{\X_{N-1}\le 3N\}\setminus \Big(A(\eps;N)\cap B(\eps,\eta;N)\Big)\biggr)\le C \exp(-N^c/C),\quad N\in\nn
\end{equation}
for suitable constants $c=c(\eps)>0$ and $C=C(\eps)<\infty$. Combining this with \eqref{eq_x7} we
conclude that the second summand tends to zero in the limit $N\to\infty$ for every fixed $\eps>0$
and any $\eta=\eta(\eps)$ as described.

\smallskip

Lastly, \eqref{eq_x15} shows that the first summand is bounded above by $C(\eta+2\eps)^{1/2}$. All in all, we have established that for every fixed $\eps>0$ and any $\eta=\eta(\eps)$ as above
$$
\limsup_{N\to\infty} \;\ev \left[\frac{1}{\X_N-\X_{N-1}}\right]\le C(\eta(\eps)+2\eps)^{1/2}.
$$
We finish the proof by taking the limit $\eps\to 0$.
\end{proof}

\smallskip

\begin{proof}[Proof of Lemma \ref{Lemma_sum_inverse_cross_level}]
In view of the interlacing of the coordinates $X_1^N,\,X^N_2,\,\ldots,\,X_N^N$ and
$X_1^{N-1},\,X_2^{N-1},\,\ldots,\,X_{N-1}^{N-1}$ the sum in Lemma
\ref{Lemma_sum_inverse_cross_level} is less or equal to the sum in Lemma \ref{lemma_sum_inverse},
but not by more than $\frac{1}{X_N^N(NT_0+t;N)-X_{N-1}^N(NT_0+t;N)}$. Hence, Lemma
\ref{Lemma_sum_inverse_cross_level} is a consequence of Lemmas \ref{lemma_sum_inverse} and
\ref{lemma_one_inverse}.
\end{proof}

\smallskip

\begin{proof}[Proof of Theorem \ref{Theorem_convergence_fixed_time}]
As before we may assume without loss of generality that $T_0=2/\beta$ and $t=0$. We start with the case $k=1$ and introduce the notations $\X=\big(\X_1,\,\X_2,\,\ldots,\,\X_N\big)$ and $\Y=\big(\Y_1,\,\Y_2,\,\ldots,\,\Y_{N-1}\big)$ for the values of the processes $\big(X_1^N,\,X^N_2,\,\ldots,\,X_N^N\big)$ and $\big(X_1^{N-1},\,X_2^{N-1},\,\ldots,\,X_{N-1}^{N-1}\big)$ at time $2N/\beta$, respectively. Moreover, we recall that the probability density of the random vector $(\X,\Y)$ is proportional to
\begin{equation}\label{eq_2_level_density}
\prod_{i=1}^N \exp\left(-\frac{\beta x_i^2}{4 N}\right) \prod_{1\le
i<j\le N} (x_j-x_i)\, \prod_{1\le i<j \le N-1} (y_j-y_i) \, \prod_{i,j} |x_i-y_j|^{\beta/2-1}.
\end{equation}
Indeed, this can be seen by combining the probability density of $\X$ in \eqref{eq_time_t_DBM}
with the conditional probability density of $\Y$ given $\X$ in \cite[Proposition 1.3]{GS}, which
is obtained by integration \eqref{eq_beta_Hermite_corners}.

\medskip

We aim to show that as the distribution of $\X_N-\Y_{N-1}$ converges to the Gamma distribution with density
$$
 \frac{\left(\frac{\beta}{2}\right)^{\beta/2}}{\Gamma(\beta/2)}\,z^{\beta/2-1}\,e^{-\frac{\beta}{2}z}
$$
in the limit $N\to\infty$. Let us condition on $\X_1=x_1,\,\X_2=x_2,\,\ldots,\,\X_N=x_N$ and $\Y_1=y_1,\,\Y_2=y_2,\,\ldots,\,\Y_{N-2}=y_{N-2}$ and study the conditional distribution of $\X_N-\Y_{N-1}$. By \eqref{eq_2_level_density} its density is proportional to
$$
f(z):=z^{\beta/2-1} \prod_{i=1}^{N-2} (x_N-y_i-z) \prod_{i=1}^{N-1} (x_N-x_i-z)^{\beta/2-1},\quad 0<z<x_N-x_{N-1}.
$$
For $x_N-x_{N-1}>1$ and $0<z<x_N-x_{N-1}$ we have
\begin{equation}
\label{eq_x17}
 \frac{f(z)}{f(1)}=z^{\beta/2-1} \prod_{i=1}^{N-2}
 \left(1-\frac{z-1}{x_N-y_i-1}\right)
 \prod_{i=1}^{N-1}\left(1-\frac{z-1}{x_N-x_i-1}\right)^{\beta/2-1}.
\end{equation}

\smallskip

In order to analyze the ratio in \eqref{eq_x17}, we combine Lemmas \ref{lemma_sum_inverse}, \ref{Lemma_inverse_expectation} and \ref{Lemma_sum_inverse_cross_level} to deduce the existence of a sequence of sets $D(N)\subset \mathbb R^{2N-2}$ such that
\[
\lim_{N\to\infty} \pp\Big((\X_1,\,\X_2,\,\ldots,\,\X_N,\,\Y_1,\,\Y_2,\,\ldots,\,\Y_{N-2})\in D(N)\Big)=1
\]
and for any sequence $D(N)\ni\big(x_1(N),x_2(N),\ldots,x_N(N),y_1(N),y_2(N),\ldots,y_{N-2}(N)\big)$, $N\in\nn$ it holds
\[
\begin{split}
\lim_{N\to\infty} \;\; \frac{1}{x_N(N)-x_{N-1}(N)}=0, \qquad
\lim_{N\to\infty} \;\; \sum_{i=1}^{N-1} \frac{1}{x_N(N)-x_{i}(N)-1}=1,\quad\text{and} \\
\lim_{N\to\infty} \;\; \sum_{i=1}^{N-2} \frac{1}{x_N(N)-y_{i}(N)-1}=1\,.
\end{split}
\]
Clearly, it suffices to study the asymptotic distribution of $\X_N-\Y_{N-1}$ on the event $(\X_1,\,\X_2,\,\ldots,\,\X_N,\,\Y_1,\,\Y_2,\,\ldots,\,\Y_{N-2})\in D(N)$ which we do from here on.

\medskip

Fix a number $L>0$. Then the following limits are uniform in $0<z<L\,$:
\begin{eqnarray*}
&& \lim_{N\to\infty} \prod_{i=1}^{N-2} \bigg(1-\frac{z-1}{x_N-y_i-1}\bigg)
 \prod_{i=1}^{N-1}\bigg(1-\frac{z-1}{x_N-x_i-1}\bigg)^{\beta/2-1}\\
&&=\lim_{N\to\infty} \exp\bigg(-\sum_{i=1}^{N-2}\frac{z-1}{x_N-y_i-1}+(1-\beta/2)\,
\sum_{i=1}^{N-1}\frac{z-1}{x_N-x_i-1}+o\bigg(\frac{1}{x_N-x_{N-1}}\bigg)\bigg)=e^{-\frac{\beta}{2}(z-1)}.
\end{eqnarray*}
This and \eqref{eq_x17} show that
\eq\label{lim_compact} \lim_{N\to\infty} \; \frac{f(z)}{f(1)}=
z^{\beta/2-1} e^{-\frac{\beta}{2} (z-1)},
\en
 which corresponds precisely to the density of the
desired Gamma distribution. To finish the proof of the case $k=1$ it remains to show that the
sequence $\X_N-\Y_{N-1}$, $N\in\nn$ is tight. In view of \eqref{lim_compact} the tightness amounts
to
\begin{equation}\label{eq_x18}
\lim_{L\to\infty}\;\limsup_{N\to\infty}\;\int_L^\infty \frac{f(z)}{f(1)}\,\mathrm{d}z=0\,.
\end{equation}
To this end, we use the elementary inequality $1-u\le e^{-u}$, $u>0$ to obtain the bounds
\begin{eqnarray*}
\frac{f(z)}{f(1)} \le z^{\beta/2-1} \exp\bigg(-\sum_{i=1}^{N-2}\frac{z-1}{x_N-y_i-1}+(1-\beta/2)
\sum_{i=1}^{N-1}\frac{z-1}{x_N-x_i-1}\bigg) \le z^{\beta/2-1} e^{-(z-1)/C}
\end{eqnarray*}
with a suitable constant  $C>0$ depending only on $\beta$. This shows \eqref{eq_x18} and finishes
the proof in the case $k=1$.

\medskip

For general $k>1$ we proceed by induction. Suppose that Theorem
\ref{Theorem_convergence_fixed_time} holds for $k=m-1$ and let us prove it for $k=m$. We need to
find the limiting conditional distribution of the spacing
$X_{N-m+1}^{N-m+1}(2N/\beta;N)-X_{N-m}^{N-m}(2N/\beta;N)$ given the spacings
$X_N^N(2N/\beta;N)-X_{N-1}^{N-1}(2N/\beta;N)$,
$X_{N-1}^{N-1}(2N/\beta;N)-X_{N-2}^{N-2}(2N/\beta;N),\,\ldots,\,X_{N-m+2}^{N-m+2}(2N/\beta;N)-X_{N-m+1}^{N-m+1}(2N/\beta;N)$.
However, if we additionally condition on $\big(X_i^j(2N/\beta;N):\;1\leq i\leq j,\;N-m+1\le j\le
N\big)$, then the conditional distribution of $\big(X_i^j(2N/\beta;N):\;1\leq i\leq
j,\;j\in\{N-m,\,N-m+1\}\big)$ is the same as in \eqref{eq_2_level_density} thanks to
\cite[Proposition 1.3]{GS}. At this point, it remains to repeat line by line the argument in the
$k=1$ case.
\end{proof}

\section{Dynamic limit theorem}

The aim of this section is to prove a dynamic version of Theorem \ref{Theorem_convergence_fixed_time}

\subsection{Statement and heuristics}

For each $k=1,2,\dots$ we let $\mathcal{C}^k$ be the space of continuous functions from $[0,\infty)$ to $[0,\infty)^k$ endowed with the topology of uniform convergence on compact sets. The space of probability measures on $\mathcal{C}^k$ admits a metric which is compatible with the topology of weak convergence, and the resulting metric space is complete and separable (see e.g. \cite[Chapter 2]{Bi}).

\begin{theorem} \label{Theorem_multitime_convergence}
For any $\beta\ge 4$, $T_0>0$ and $k=1,2,\dots$ the processes
\eq\label{eq_differences_vector}
\big(X_N^N(NT_0+t;N)-X_{N-1}^{N-1}(NT_0+t;N),\ldots,X^{N-k+1}_{N-k+1}(NT_0+t;N)-X^{N-k}_{N-k}(NT_0+t;N)\big),\;\; t\ge 0
\en
converge in distribution on $\mathcal{C}^k$. Moreover, the law of the limiting process is that of the unique weak solution $\big(R_1(t),\,R_2(t),\,\ldots,\,R_k(t)\big)$, $t\ge0$ to the system of SDEs
\begin{equation}
\label{eq_SDE_for_diffs}
\begin{split}
& \mathrm{d}R_i(t)=\frac{(\beta/2-1)\,\mathrm{d}t}{R_i(t)}\,
\mathrm{d}t-\frac{(\beta/2-1)\,\mathrm{d}t}{R_{i+1}(t)}+\mathrm{d}B_i(t)-\mathrm{d}B_{i+1}(t),\quad i=1,2,\ldots,k-1,\\
& \mathrm{d}R_k(t)=\frac{(\beta/2-1)\,\mathrm{d}t}{R_k(t)}-\sqrt{\frac{\beta}{2
T_0}}\,\mathrm{d}t+\mathrm{d}B_k(t)-\mathrm{d}B_{k+1}(t),
\end{split}
\end{equation}
where $B_1,\,B_2,\,\ldots,\,B_{k+1}$ are i.i.d.\ one-dimensional standard Brownian motions, and
the initial condition is chosen according to the product Gamma distribution of Theorem
\ref{Theorem_convergence_fixed_time}.
\end{theorem}

At this point, Theorems \ref{Theorem_multitime_convergence_intro} and \ref{Theorem_multitime_convergence_intro_2} follow from Theorems \ref{Theorem_convergence_fixed_time} and \ref{Theorem_multitime_convergence}, together with the fact that \eqref{eq_SDE_for_diffs} is the system of SDEs for the differences of coordinates in \eqref{eq_SDE_for_coords} and the weak uniqueness for these two systems of SDEs (established in Theorems \ref{Theorem_uniqueness}, \ref{Theorem_uniqueness_2} below).

\medskip

The validity of Theorem \ref{Theorem_multitime_convergence} for all $k$ implies the following
curious property of the process $\big(R_1(t),\,R_2(t),\,\ldots,\,R_k(t)\big)$, $t\ge0$ for any
fixed $k$: under the product Gamma initial condition of Theorems
\ref{Theorem_convergence_fixed_time} and \ref{Theorem_multitime_convergence} the process of the
first $(k-1)$ coordinates $\big(R_1(t),\,R_2(t),\,\ldots,\,R_{k-1}(t)\big)$ solves a closed system
of SDEs of the form \eqref{eq_SDE_for_diffs} (with $k$ replaced by $(k-1)$). In Section
\ref{Section_Limit_properties} we discuss how a direct proof of this fact could be obtained.

\medskip

In the rest of this subsection we give an informal argument explaining the validity of Theorem \ref{Theorem_multitime_convergence}. The following three subsections are then devoted to a rigorous proof of that theorem.

\medskip

According to Theorem \ref{Theorem_GS_multilevel} and Proposition \ref{Proposition_GS_restriction} the process of \eqref{eq_differences_vector} satisfies the SDEs
\begin{equation} \label{eq_prelimit_SDE}
\begin{split}
\mathrm{d} \bigl( X_{N-a}^{N-a}(NT_0+t;N)-X_{N-a-1}^{N-a-1}(NT_0+t;N)\bigr)
=\mathrm{d} W^{N-a}_{N-a}(NT_0+t)- \mathrm{d} W^{N-a-1}_{N-a-1}(NT_0+t) \\
+\frac{(\beta/2-1)\,\mathrm{d}t}{X_{N-a}^{N-a}(NT_0+t;N)-X_{N-a-1}^{N-a-1}(NT_0+t;N)}
-\frac{(\beta/2-1)\,\mathrm{d}t}{X_{N-a-1}^{N-a-1}(NT_0+t;N)-X_{N-a-2}^{N-a-2}(NT_0+t;N)} \\
+ S_a(t;N)\,\mathrm{d}t
\end{split}
\end{equation}
for $a=0,\,1,\,\ldots,\,k-2$, and
\begin{equation}
\begin{split}
\label{eq_prelimit_SDE_2} \mathrm{d} \bigl( X_{N-k+1}^{N-k+1}(NT_0+t;N)-X_{N-k}^{N-k}(NT_0+t;N)
\bigr)=\mathrm{d} W^{N-k+1}_{N-k+1}(NT_0+t)- \mathrm{d} B(NT_0+t) \\
+\frac{(\beta/2-1)\,\mathrm{d}t}{X_{N-k+1}^{N-k+1}(NT_0+t;N)-X_{N-k}^{N-k}(NT_0+t;N)} +
\hat S_k(t;N)\,\mathrm{d}t.
\end{split}
\end{equation}
Here $W_N^N,\,W^{N-1}_{N-1},\,\ldots,\,W_{N-k}^{N-k}$ and $B$ are i.i.d.\ standard Brownian
motions,
\begin{equation} \label{eq_S_remainder}
\begin{split}
&S_a(t;N)=-\sum_{i=1}^{N-a-1}\frac{\beta/2-1}{X_{N-a}^{N-a}(NT_0+t;N)-X_{N-a}^{i}(NT_0+t;N)}\\
&+\sum_{i=1}^{N-a-2} \Big(\frac{\beta/2-1}{X_{N-a}^{N-a}(NT_0+t;N)-X_{N-a-1}^{i}(NT_0+t;N)}
-\frac{\beta/2-1}{X_{N-a-1}^{N-a-1}(NT_0+t;N)-X_{N-a-1}^{i}(NT_0+t;N)} \Big)\\
&+\sum_{i=1}^{N-a-3}\frac{\beta/2-1}{X_{N-a-1}^{N-a-1}(NT_0+t;N)-X_{N-a-2}^{i}(NT_0+t;N)}\,,
\end{split}
\end{equation}
$a=0,\,1,\,\ldots,\,k-2$, and
\begin{equation} \label{eq_S_hat_remainder}
\begin{split}
&\hat S_k(t;N)=-\sum_{i=1}^{N-k}\frac{\beta/2-1}{X_{N-k+1}^{N-k+1}(NT_0+t;N)-X_{N-k+1}^{i}(NT_0+t;N)}\\
&+\sum_{i=1}^{N-k-1} \Big(\frac{\beta/2-1}{X_{N-k+1}^{N-k+1}(NT_0+t;N)-X_{N-k}^{i}(NT_0+t;N)}
-\frac{\beta/2}{X_{N-k}^{N-k}(NT_0+t;N)-X_{N-k}^{i}(NT_0+t;N)}\Big).
\end{split}
\end{equation}
For \textit{fixed} $t\ge0$ the expressions in \eqref{eq_S_remainder}, \eqref{eq_S_hat_remainder} can be analyzed using Lemmas
\ref{lemma_sum_inverse}, \ref{Lemma_sum_inverse_cross_level} which yield
$$
\lim_{N\to\infty} S_a(t;N)=0,\;a=0,\,1,\,\ldots,\,k-2, \quad\text{and}\quad
\lim_{N\to\infty} \hat S_k(t;N)=\sqrt{\frac{\beta}{2 T_0}}.
$$
This suggests that, as $N\to\infty$, a solution of the SDEs \eqref{eq_prelimit_SDE},
\eqref{eq_prelimit_SDE_2} should converge to a solution of the SDE \eqref{eq_SDE_for_diffs}. To
give a rigorous proof of this conclusion, we proceed in the following steps: in Section
\ref{Section_tightness} we show that the sequence of processes in \eqref{eq_differences_vector} is
tight on $\mathcal{C}^k$; then, in Section \ref{Section_uniqueness} we prove weak uniqueness for
the system of SDEs \eqref{eq_SDE_for_diffs}; finally, in Section \ref{Section_convergence} we show
that any limit point of the sequence in \eqref{eq_differences_vector} is a weak solution of
\eqref{eq_SDE_for_diffs}. Together these three steps give Theorem
\ref{Theorem_multitime_convergence}.

\subsection{Tightness}\label{Section_tightness}

In this section we prove the tightness of the sequence of processes in \eqref{eq_differences_vector}. Since the tightness of a sequence of $\rr^k$-valued processes is equivalent to the tightness of the sequences of their components, we need to prove the following statement.

\begin{proposition} \label{Proposition_tightness}
For any $\beta\ge 4$ and $T_0\ge 0$ the sequence of processes
$$X_N^N(NT_0+t;N)-X_{N-1}^{N-1}(NT_0+t;N),\;\; t\ge 0,$$
indexed by $N\in\nn$, is tight on $\mathcal C^1$.
\end{proposition}

\begin{proof}
To simplify the exposition we introduce the notation $\R^N(t):=X^{N}_{N}(NT_0+t)-X^{N-1}_{N-1}(NT_0+t)$, $t\ge0$. To prove Proposition \ref{Proposition_tightness} we aim to apply the criterion for tightness of \cite[Corollary 3.7.4]{EK} and need to prove the following two statements:
\begin{enumerate}
\item For every fixed $t\ge0$ the sequence of random variables $\R^N(t)$, $N=1,\,2,\,\ldots$ is tight on $\mathbb R$.
\item For every fixed $T>0$ and $\Delta>0$ there exists a $\delta>0$ such that
\begin{equation} \label{eq_tightness2}
\limsup_{N\to\infty} \; \pp\Big(\sup_{0\le t_1<t_2\le T,\,t_2-t_1<\delta} |\R^N(t_2)-\R^N(t_1)|>\Delta\Big)<\Delta.
\end{equation}
\end{enumerate}
The first statement is a consequence of Theorem \ref{Theorem_convergence_fixed_time}, so that from now on we fix $T>0$ and $\Delta>0$ and focus on the proof of \eqref{eq_tightness2}.

\bigskip

Theorem \ref{Theorem_GS_multilevel} and Proposition \ref{Proposition_GS_restriction} show that the
process $(X^{N}_{N},X^{N-1}_{N-1})$ solves the following system of SDEs in the filtration
generated by the processes $(X^{N}_1,X^{N}_2,\ldots,X^{N}_{N})$ and
$(X^{N-1}_1,X^{N-1}_2,\ldots,X^{N-1}_{N-1})$:
\begin{eqnarray}
&& \mathrm{d}X^{N-1}_{N-1}(t;N)=\frac{\beta}{2}\,\sum_{j=1}^{N-2} \frac{\mathrm{d}t}{X^{N-1}_{N-1}(t;N)-X^{N-1}_j(t;N)}+\mathrm{d}B_1(t), \\
&& \mathrm{d}X^{N}_{N}(t;N)=\sum_{j=1}^{N-1}
\frac{(\beta/2-1)\,\mathrm{d}t}{X^{N}_{N}(t;N)-X^{N-1}_j(t;N)} -\sum_{j=1}^{N-1}
\frac{(\beta/2-1),\mathrm{d}t}{X^{N}_{N}(t;N)-X^{N}_j(t;N)}+\mathrm{d}B_{2}(t)
\end{eqnarray}
with $B_{1}$, $B_{2}$ being suitable independent one-dimensional standard Brownian motions. Consequently,
\begin{equation} \label{eq_R_SDE}
\begin{split}
\mathrm{d}\R^N(t)&=\frac{(\beta/2-1)\,\mathrm{d}t}{\R^N(t)}+\sqrt{2}\,\mathrm{d}B_3(t)
-\frac{\beta}{2}\,\sum_{j=1}^{N-1} \frac{\mathrm{d}t}{X^{N}_{N}(NT_0+t;N)-X^{N}_j(NT_0+t;N)} \\
&\;\;\;+\sum_{j=1}^{N-2} \frac{(\beta/2-1)\,\mathrm{d}t}{X^{N}_{N}(NT_0+t;N)-X^{N-1}_j(NT_0+t;N)}
-\sum_{j=1}^{N-1} \frac{(\beta/2-1)\,\mathrm{d}t}{X^{N}_{N}(NT_0+t;N)-X^{N}_j(NT_0+t;N)}
\end{split}
\end{equation}
with a one-dimensional standard Brownian motion $B_3$.

\medskip

Next, we cover the interval $[0,T]$ by intervals $I_m:=\big[m\delta,\min((m+2)\delta,T)\big]$, $m=0,\,1,\,\ldots,\,\lfloor T/\delta \rfloor$ and use the triangle inequality together with the union bound to estimate the probability in \eqref{eq_tightness2} from above by
\begin{equation}\label{eq_tightUBD} \sum_{m=0}^{\lfloor T/\delta \rfloor} \; \pp\Big(\sup_{t\in I_m}
\big(\R^N(t)-\R^N(m\delta)\big)>\Delta/2\Big) + \sum_{m=0}^{\lfloor T/\delta \rfloor} \;
\pp\Big(\inf_{t\in I_m} \big(\R^N(t)-\R^N(m\delta)\big)<-\Delta/2\Big).
\end{equation}

\smallskip

To bound the first sum further we choose an integer $D\ge \beta/2$ and introduce for each $m$ the
auxiliary process $\widetilde{\R}^N_{m}$ given by the unique strong solution of
\begin{equation}
\label{eq_R_modified_SDE} \mathrm{d}\widetilde{\R}^N_{m}(t)=\frac{(D-1)\,
\mathrm{d}t}{\widetilde{\R}^N_{m}(t)}+\sqrt{2}\,\mathrm{d}B_{3}(t),\;\; t\ge m\delta, \qquad
\widetilde{\R}^N_{m}(m\delta)=\R^N(m\delta)
\end{equation}
where $B_3$ is the Brownian motion from \eqref{eq_R_SDE}. Dividing \eqref{eq_R_modified_SDE} by
$\sqrt{2}$ reveals that $\widetilde{\R}^N_{m}$, $t\ge m\delta$ is the $\sqrt{2}$ multiple of a
Bessel process of dimension $D$ started from $2^{-1/2}\,\R^N(m\delta)$ (see e.g. \cite[Chapter
XI]{RY} for the definition and properties of Bessel processes).

\medskip

We claim that the inequality $\widetilde{\R}^N_{m}(t)\ge \R^N(t)$ holds for all $t\ge m\delta$
with probability one. Indeed, note that due to interlacing of the coordinates $X^j_i$, $1\le i\le
j\le N$:
\begin{equation}
 \sum_{j=1}^{N-2} \frac{(\beta/2-1)}{X^{N}_{N}(NT_0+t;N)-X^{N-1}_j(NT_0+t;N)}
-\sum_{j=1}^{N-1} \frac{(\beta/2-1)}{X^{N}_{N}(NT_0+t;N)-X^{N}_j(NT_0+t;N)}\le 0, \label{pos1}
\end{equation}
\begin{equation}
 -\frac{\beta}{2}\,\sum_{j=1}^{N-1} \frac{1}{X^{N}_{N}(NT_0+t;N)-X^{N}_j(NT_0+t;N)}\le 0.
\label{pos2}
\end{equation}
This allows to derive the desired comparison inequality between $\widetilde{\R}^N_{m}$ and $\R^N$ by arguing in the spirit of the classical comparison theorems for SDEs (see e.g. \cite[Chapter IX, Theorem 3.7]{RY} and \cite[Chapter 5, Proposition 2.18]{KS}). Indeed, define the stopping times
$$
\tau_\eps:=\inf\{t\ge m\delta:\;\widetilde{\R}^N_{m}(t)\le\R^N(t)-\eps\},\quad\eps>0.
$$
Due to the almost sure continuity of the trajectories of $\widetilde{\R}^N_{m}$ and $\R^N$ we
conclude that if $\tau_\eps$ is finite, then
$\widetilde{\R}^N_{m}(\tau_\eps)=\R^N(\tau_\eps)-\eps$. Moreover, on the event $\tau_\eps<\infty$
we can use \eqref{eq_R_SDE} and \eqref{eq_R_modified_SDE} to write for $t\ge\tau_\eps$
\begin{equation} \label{eq_x20}
\begin{split}
&\widetilde{\R}^N_{m}(t)-\R^N(t)=\int_{\tau_\eps}^t\bigg(\frac{D-1}{\widetilde{\R}^N_m(s)}-\frac{\beta/2-1}{\R^N(s)}
+\frac{\beta}{2} \sum_{j=1}^{N-1}\frac{1}{X^{N}_{N}(NT_0+s)-X^{N}_j(NT_0+s)} \\
&\quad\quad\quad\quad\quad\;-\bigg(\sum_{j=1}^{N-2} \frac{(\beta/2-1)}{X^{N}_{N}(NT_0+s)-X^{N-1}_j(NT_0+s)}
 -\sum_{j=1}^{N-1} \frac{(\beta/2-1)}{X^{N}_{N}(NT_0+s)-X^{N}_j(NT_0+s)}\bigg)\bigg)\,\mathrm{d}s.
\end{split}
\end{equation}
 However, for small $(t-\tau_\eps)$ the left-hand side of \eqref{eq_x20} is
negative, while the right-hand side of \eqref{eq_x20} is positive in view of
\eqref{pos1},\eqref{pos2} and the non-negativity of the processes $\widetilde{\R}^N_{m}$ and
$\R^N$. This contradiction proves that none of the stopping times $\tau_\eps$, $\eps>0$ can be
finite with positive probability.

\medskip

Thanks to the established comparison result we can now bound the first sum in \eqref{eq_tightUBD} by
\begin{equation}
\label{eq_x21} \sum_{m=0}^{\lfloor T/\delta \rfloor} \; \pp\Big(\sup_{t\in I_m}
\big(\widetilde{\R}^N_{m}(t)-\widetilde{\R}^N_{m}(m\delta)\big)>\Delta/2\Big).
\end{equation}
Moreover, we recall that the Bessel process of dimension $D$ describes the evolution of the
Euclidean norm of a $D$-dimensional standard Brownian motion, and the triangle inequality for the
Euclidean norm:
$$
\big|(a_1^2+a_2^2\dots+a_D^2)^{1/2}-(b_1^2+\dots+b_D^2)^{1/2}\big|\le \big(|a_1-b_1|^2+|a_2-b_2|^2+\ldots+|a_D-b_D|^2\big)^{1/2}.
$$
This allows to bound the expression in \eqref{eq_x21} further by
\begin{equation}
\label{eq_x22} \big(1+\lfloor T/\delta \rfloor\big)\cdot D\cdot\pp\Big(\sup_{0<t<2\delta}
|B(t)|>\frac{\Delta}{2\sqrt{2D}}\Big),
\end{equation}
where $B$ is a one-dimensional standard Brownian motion (so that, in particular, $B(0)=0$). Using
the union bound and the explicitly known distribution of the running maximum of a standard
Brownian motion (see e.g. \cite[Section 2.8.A]{KS}) we see that, for any $\Delta>0$, the
expression in \eqref{eq_x22} tends to $0$ in the limit $\delta\to 0$. Therefore, for small enough
$\delta$ the first sum in \eqref{eq_tightUBD} is less than $\Delta/2$.

\medskip

To estimate the second sum in \eqref{eq_tightUBD} we first note that the interlacing of the
coordinates $X^j_i$, $1\le i\le j\le N$ implies
\begin{multline*}
 \frac{(\beta/2-1)}{\R^N}+ \sum_{j=1}^{N-2} \frac{(\beta/2-1)}{X^{N}_{N}(NT_0+t;N)-X^{N-1}_j(NT_0+t;N)}
-\sum_{j=1}^{N-1} \frac{(\beta/2-1)}{X^{N}_{N}(NT_0+t;N)-X^{N}_j(NT_0+t;N)}\\ \ge 0.
\end{multline*}
This and \eqref{eq_R_SDE} allow to bound the second sum in \eqref{eq_tightUBD} by
\begin{multline}\label{eq_tightUBD2}
\sum_{m=0}^{\lfloor T/\delta \rfloor}\,\Biggl[\pp\Big(\inf_{t\in I_m}
\big(B_3(t)-B_3(m\delta)\big)<-\frac{\Delta}{4}\Big) \\ + \pp\bigg(\int_{I_m} \sum_{j=1}^{N-2}
\frac{\beta/2\,\mathrm{d}t}{X^{N-1}_{N-1}(NT_0+t;N)-X^{N-1}_j(NT_0+t;N)} >
\frac{\Delta}{4}\bigg)\Biggr].
\end{multline}
The first type of summands in \eqref{eq_tightUBD2} can be again estimated using the explicit distribution of the running maximum of a standard Brownian motion, which reveals that their sum can be made smaller than $\Delta/4$ by choosing a $\delta$ small enough.

\smallskip

Finally, we bound the sum of the second type of summands in \eqref{eq_tightUBD2} by applying successively Markov's inequality, Jensen's inequality, Fubini's theorem and Lemma \ref{Lemma_inverse_expectation_squared}:
\begin{eqnarray*}
&& \sum_{m=0}^{\lfloor T/\delta \rfloor}\;\pp\bigg(\int_{I_m} \frac{\beta}{2}\,\sum_{j=1}^{N-2}
\frac{\mathrm{d}t}{X^{N-1}_{N-1}(NT_0+t;N)-X^{N-1}_j(NT_0+t;N)} > \frac{\Delta}{4}\bigg) \\
&\le& \frac{16}{\Delta^2}\,\sum_{m=0}^{\lfloor T/\delta \rfloor}\; \ev\bigg[\bigg(\int_{I_m}
\frac{\beta}{2}\,\sum_{j=1}^{N-2} \frac{\mathrm{d}t}{X^{N-1}_{N-1}(NT_0+t;N)-X^{N-1}_j(NT_0+t;N)}\bigg)^2\bigg] \\
&\le& \frac{16}{\Delta^2}\,\sum_{m=0}^{\lfloor T/\delta \rfloor}\;\ev\bigg[|I_m|\,\int_{I_m}
\frac{\beta^2}{4}\;
\bigg(\sum_{j=1}^{N-2} \frac{1}{X^{N-1}_{N-1}(NT_0+t;N)-X^{N-1}_j(NT_0+t;N)}\bigg)^2\,\mathrm{d}t\bigg] \\
&=& \frac{16}{\Delta^2}\,\sum_{m=0}^{\lfloor T/\delta \rfloor}\;|I_m|\,\int_{I_m}
\frac{\beta^2}{4} \;\ev\bigg[\bigg(\sum_{j=1}^{N-2}
\frac{1}{X^{N-1}_{N-1}(NT_0+t;N)-X^{N-1}_j(NT_0+t;N)}\bigg)^2\bigg]\,\mathrm{d}t\\
 &\le&
C\,\frac{16}{\Delta^2}\,\sum_{m=0}^{\lfloor T/\delta \rfloor} |I_m|^2
\end{eqnarray*}
where $C>0$ is a suitable constant. Since $|I_m|=2\delta$, we see that the latter upper bound
tends to $0$ in the limit $\delta\to 0$. In particular, for small enough $\delta$ it is less than
$\Delta/4$.

\smallskip

All in all, we have shown that for $\delta$ small enough the expression in \eqref{eq_tightUBD} is less than $\Delta/2+\Delta/4+\Delta/4=\Delta$ for all $N$.
\end{proof}

\subsection{Limiting SDE}\label{Section_uniqueness}

The goal of this section is to prove the following theorem.

\begin{theorem} \label{Theorem_uniqueness}
For any $k=1,2,\dots$ and any initial condition $R(0)\in(0,\infty)^k$ the system of SDEs \eqref{eq_SDE_for_diffs} possesses a unique weak solution taking values in $[0,\infty)^k$. The solution $R$ is a Markov process and satisfies $R_i(t)>0$ for all $t\ge 0$ and $1\le i\le k$ with probability one.
\end{theorem}

Our proof of Theorem \ref{Theorem_uniqueness} is based on a Girsanov change of measure that will simplify the SDEs in consideration. We refer the reader to \cite[Section 3.5]{KS} and \cite[Section 5.3]{KS} for general information about Girsanov's theorem and weak solutions of SDEs. We fix a $T>0$ and will establish all claims of Theorem \eqref{Theorem_uniqueness} on the time interval $[0,T]$. Clearly, then the theorem will follow from the arbitrariness of $T$.

\medskip

Take a $\delta\ge 0$. For a $[0,\infty)^k$-valued stochastic process $Y(t)=(Y_1(t),\dots,Y_k(t))$,
$t\in[0,T]$ define the stopping time
\begin{equation}
\label{eq_stop_time_def} \tau_\delta[Y]:=\inf\big\{t\in[0,T]:\;Y_i(t)\le\delta\;\text{ for some
}i\big\},
\end{equation}
 with the convention $\inf\,\emptyset=\infty$. Our first aim is to analyze
$\tau_0[R]$.

\begin{lemma} \label{lemma_no_zero}
Let $R$ be a weak solution of \eqref{eq_SDE_for_diffs} with a deterministic initial condition $R(0)\in(0,\infty)^k$. Then $\tau_0[R]=\infty$ almost surely.
\end{lemma}

\begin{proof}
We argue by induction over $k$. For $k=1$ the SDE is one-dimensional and its integral form reads
\begin{equation}\label{eq_SDE_k1}
R_1(t)=R_1(0)+\int_0^t\frac{(\beta/2-1)\,\mathrm{d}s}{R_1(s)}-\sqrt{\frac{\beta}{2T_0}}\,t+B_1(t)-B_{2}(t).
\end{equation}
By the Girsanov theorem (see e.g. \cite[Theorem 5.1, Chapter 3]{KS}) there exists a probability measure $\widetilde P$ equivalent to the underlying probability measure such that
$$
-\sqrt{\frac{\beta}{2T_0}}\,t+B_1(t)-B_2(t),\;\;t\in[0,T]
$$
is a standard Brownian motion under $\widetilde P$. Consequently, under $\widetilde P$ the process $2^{-1/2}\,R_1$ solves the SDE for the Bessel process of dimension $\beta/2$. Since $\beta\ge4$, it follows from the results of \cite[Section XI.1]{RY} that the latter process does not reach zero with probability one. Thus, $R_1$ does not reach zero before time $T$ under $\widetilde P$ with probability one, and the same is true under the original probability measure thanks to the equivalence of the two measures.

\medskip

Now, we consider an arbitrary $k>1$. By the induction hypothesis the coordinates $R_2,\,R_3,\,\ldots,\,R_k$ do not reach zero by time $T$, and it remains to analyze $R_1$. To this end, pick a $\delta>0$ and let
\[
\rho_\delta:=\inf\big\{t\ge0:\;R_2(t)\le\delta\big\}.
\]
Similarly to the case $k=1$ we can move to an equivalent measure $\widetilde P$ (noting that Novikov's condition \cite[Corollary 5.13, Chapter 3]{KS} is satisfied) such that under $\widetilde P$ the process
\[
-\int_0^{t\wedge\rho_\delta}\,\frac{(\beta/2-1)\,\mathrm{d}s}{R_2(s)}+B_1(t\wedge\rho_\delta)-B_2(t\wedge\rho_\delta),
\;\;t\in[0,T]
\]
is a Brownian motion stopped at $\rho_\delta$. Consequently, up to time $\rho_\delta$ the process $2^{-1/2}\,R_1$ under $\widetilde P$ coincides pathwise with a Bessel process of dimension $\beta/2$, and therefore does not reach zero up to time $\rho_\delta$ under either of the two measures. It remains to pass to the limit $\delta\to0$ and to invoke the induction hypothesis to deduce $\lim_{\delta\to 0}\rho_\delta=\infty$ under the original probability measure.
\end{proof}

\smallskip

Next, we turn to the uniqueness part of Theorem \ref{Theorem_uniqueness}. We let $R$ be a weak solution of \eqref{eq_SDE_for_diffs} with a given initial condition and observe that up to time $\tau_\delta[R]$ all drifts in the SDEs of \eqref{eq_SDE_for_diffs} are bounded. Hence, we can apply a Girsanov change of measure (noting that Novikov's condition \cite[Corollary 5.13, Chapter 3]{KS} is satisfied due to the boundedness of the integrands in the stochastic exponential) with a density of the form
$$
\exp\bigg(\sum_{i=1}^{k+1} \bigg(\int_0^T b_i(t)\,\mathbf{1}_{\{t\le \tau_\delta[R]\}}\,\mathrm{d}B_i(t)
- \frac{1}{2}\int_0^T b_{i}(t)^2\,\mathbf{1}_{\{t\le\tau_\delta[R]\}}\,\mathrm{d}t\bigg)\bigg)
$$
such that under the new measure $\widetilde P$
\begin{equation}\label{eq_modified_SDE}
 R_i(t\wedge \tau_\delta[R])=R_i(0)+\int_{0}^{t\wedge\tau[R]} \frac{\beta/2-1}{R_i(s)}\,\mathrm{d}s + \widetilde
 B_i(t\wedge \tau_\delta[R])-\widetilde B_{i+1}(t\wedge\tau_\delta[R]),\;\;\; 1\le i\le k
\end{equation}
on $[0,T]$, with $\widetilde B_i$, $1\le i\le k+1$ being i.i.d.\ one-dimensional standard Brownian
motions under $\widetilde P$.

\medskip

The equations of \eqref{eq_modified_SDE} imply that under $\widetilde P$ each $R_i$ is the $\sqrt{2}$ multiple of the Bessel process of dimension $\beta/2$ driven by the Brownian motion $2^{-1/2}(\widetilde B_i-\widetilde B_{i+1})$. The pathwise uniqueness of the Bessel process (see e.g. \cite[Chapter XI, Section 1]{RY}) shows that the joint law of the processes $R_i(t\wedge\tau_\delta[R])$, $t\in[0,T]$ with $i=1,2,\ldots,k$ and the stopping time $\tau_\delta[R]$ is uniquely determined under $\widetilde P$. Making a Girsanov change of measure back to the original probability measure we conclude that the joint law of
the process $R(t\wedge\tau_\delta[R])$, $t\in[0,T]$ and the stopping time $\tau_\delta[R]$ is also uniquely determined under the original probability measure (a detailed version of this argument can be found e.g. in the proof of \cite[Proposition 5.3.10]{KS}). Since $\delta$ was arbitrary and all components of $R$ have continuous paths, it follows that the joint law of the process $R(t\wedge\tau_0[R])$, $t\in[0,T]$ and the stopping time $\tau_0[R]$ is uniquely determined under the original probability measure as well. It remains to note that $\tau_0[R]=\infty$ almost surely by Lemma \ref{lemma_no_zero}.

\medskip

For the existence part of Theorem \ref{Theorem_uniqueness} we argue similarly. Pick $(k+1)$ independent standard Brownian motions $\widetilde B_i$ $1\le i\le k+1$ and define the process $R$ as the strong solution of
\begin{equation}\label{eq_Bessel_SDE}
R_i(t)=R_i(0)+\int_0^t \frac{\beta/2-1}{R_i(s)}\,\mathrm{d}s + \widetilde B_i(t)-\widetilde B_{i+1}(t),\quad\;\;\; 1\le i\le k.
\end{equation}
The existence of such a solution follows from the corresponding existence theorem for Bessel processes (see e.g. \cite[Chapter XI, Section 1]{RY}). Now, for any $\delta>0$ we can make a Girsanov change of measure, so that under the new measure $\widehat P$ the process $R$ satisfies \eqref{eq_SDE_for_diffs} with suitable Brownian motions up to time $\tau_\delta[R]$, that is
\begin{equation}\label{eq_SDE_for_diffs_stopped}
\begin{split}
R_i(t\wedge \tau_\delta[R])=R_i(0)+\int_0^{t\wedge \tau_\delta[R]}
\bigg(\frac{\beta/2-1}{R_i(s)}-\frac{\beta/2-1}{R_{i+1}(s)}\bigg)\,\mathrm{d}s+B_i(t\wedge \tau_\delta[R])-B_{i+1}(t\wedge \tau_\delta[R]),\\
i=1,\,2,\,\ldots,\,k-1, \\
R_k(t\wedge \tau_\delta[R])=R_k(0)+\int_0^{t\wedge \tau_\delta[R]}
\bigg(\frac{\beta/2-1}{R_k(s)}-\sqrt{\frac{\beta}{2 T_0}}\bigg)\,\mathrm{d}s
+B_k(t\wedge\tau_\delta[R])-B_{k+1}(t\wedge \tau_\delta[R])
\end{split}
\end{equation}
where $B_1,\,B_2,\,\ldots,\,B_{k+1}$ are i.i.d.\ one-dimensional standard Brownian motions. Thanks
to the uniqueness part we know that the solutions of \eqref{eq_SDE_for_diffs_stopped} are
consistent in the sense of the Kolmogorov Extension Theorem (see e.g.  in \cite[Theorem
6.16]{Ka}). Passing to the limit $\delta\to 0$ we therefore obtain a solution of
\eqref{eq_SDE_for_diffs_stopped} with $\delta=0$. Moreover, $\tau_0[R]=\infty$ with probability
one by Lemma \ref{lemma_no_zero}, so that in fact we have constructed a weak solution of
\eqref{eq_SDE_for_diffs}.

\medskip

At this point, we also remark that the just established existence and uniqueness in law of the solution $R$ to \eqref{eq_SDE_for_diffs} implies that $R$ is a Markov process via standard arguments (see e.g \cite[Theorem 5.4.20]{KS} and its proof).

\medskip

Repeating the same proof line by line one also arrives at the following theorem.

\begin{theorem} \label{Theorem_uniqueness_2}
For any $k=1,\,2,\,\ldots$ and any initial condition $Z_1^{(k)}(0)>Z_2^{(k)}(0)>\cdots>Z_{k+1}^{(k)}(0)$ the system of SDEs \eqref{eq_SDE_for_coords} possesses a unique weak solution. The solution is a Markov process and satisfies $Z_1^{(k)}(t)>Z_2^{(k)}(t)>\cdots>Z_{k+1}^{(k)}(t)$ for all $t\ge0$ with probability one.
\end{theorem}

\subsection{Convergence to the limit}\label{Section_convergence}

In this subsection we will show that every limit point of the sequence in
\eqref{eq_differences_vector} is a weak solution of the system of SDEs \eqref{eq_SDE_for_diffs}.
Let $R=(R_1,\,R_2,\,\ldots,\,R_k)$ be an arbitrary such limit point. As will become apparent from
the following argument we may assume without loss of generality that $R$ is the limit of the whole
sequence in \eqref{eq_differences_vector}. In addition, thanks to the Skorokhod Representation
Theorem in the form of \cite[Theorem 3.5.1]{Du} we may also assume that the processes in
\eqref{eq_differences_vector} are all defined on the same probability space
$(\Omega,\mathcal{F},\pp)$ and converge to $R$ in the almost sure sense with respect to the
topology on $\mathcal{C}^k$. In other words, with the notation
\begin{equation} \label{eq_diff_notation}
\R^N_i(t):=X^{N+1-i}_{N+1-i}(NT_0+t;N)-X^{N-i}_{N-i}(NT_0+t;N),\;\;t\ge0,\quad
i=1,\,2,\,\ldots,\,k
\end{equation}
we have
\begin{equation}
\label{eq_convergence_of_components}
\lim_{N\to\infty} \R^N_i = R_i,\quad i=1,\,2,\,\ldots,\,k
\end{equation}
uniformly on compact sets. As announced we will now prove the following.

\begin{proposition} \label{Proposition_limit_SDE}
The process $R$ is a weak solution of the system of SDEs \eqref{eq_SDE_for_diffs}.
\end{proposition}

\smallskip

Fix a $\delta>0$ and let $\F_\delta$ be the set of infinitely differentiable functions
$f:\,[0,\infty)^k\to\rr$ with support contained in the set
$$
S_\delta:=\big\{(x_1,\,x_2,\,\ldots,\,x_k)\in[0,\infty)^k:\;\delta<x_i<\delta^{-1} \text{ for all
}i=1,\,2,\,\ldots,\,k\big\}.
$$
For functions $f\in\F_\delta$ consider the process
\eq\label{eq_martingale}
\begin{split}
M^f(t):=f(R(t))-f(R(0))
-\int_0^t \,\sum_{i=1}^{k-1} \bigg(\frac{\beta/2-1}{R_i(s)}-\frac{\beta/2-1}{R_{i+1}(s)}\bigg)
\,\frac{\partial f}{\partial x_i}(R(s))\,\mathrm{d}s \quad\quad\quad\quad\quad \\
-\int_0^t \bigg(\frac{\beta/2-1}{R_k(s)}-\sqrt{\frac{\beta}{2T_0}}\bigg)\,\frac{\partial f}{\partial x_k}(R(s))
+\sum_{i,j=1}^k a_{i,j}\,\frac{\partial^2 f}{\partial x_i\partial x_j}(R(s))\,\mathrm{d}s,\;\;t\ge0
\end{split}
\en
where $a_{i,j}=\mathbf{1}_{\{i=j\}}-\frac{1}{2}\,\mathbf{1}_{\{|i-j|=1\}}$, $i,j\in\{1,2,\ldots,k\}$.

\begin{lemma} \label{Lemma_limit_martingale}
For any $\delta>0$ and $f\in\F_\delta$ the process $M^f$ is a continuous martingale in its natural filtration.
\end{lemma}

\begin{proof}
For each $N=1,\,2,\,\ldots$ define the process
\begin{equation} \label{eq_martingale_approximation}
\begin{split}
M^f_{N}(t):= f(\R^N(t))-f(\R^N(0)) \\-\int_0^t \sum_{i=1}^{k-1} \frac{\partial f}{\partial
x_i}(\R^N(s))
\Biggl(\sum_{j=1}^{N-i} \frac{\beta/2-1}{X^{N-i+1}_{N-i+1}(NT_0+s;N)-X^{N-i}_j(NT_0+s;N)} \\
-\sum_{j=1}^{N-i} \frac{\beta/2-1}{X^{N-i+1}_{N-i+1}(NT_0+s;N)-X^{N-i+1}_j(NT_0+s;N)}\\
-\sum_{j=1}^{N-i-1} \frac{\beta/2-1}{X^{N-i}_{N-i}(NT_0+s;N)-X^{N-i-1}_j(NT_0+s;N)} \\
+\sum_{j=1}^{N-i-1} \frac{\beta/2-1}{X^{N-i}_{N-i}(NT_0+s;N)-X^{N-i}_j(NT_0+s;N)}\Biggr)\,\mathrm{d}s \\
-\int_0^{t} \frac{\partial f}{\partial x_n}(\R^N(s)) \Biggl(-\sum_{j=1}^{N-k-1}
\frac{\beta/2}{X^{N-k}_{N-k}(NT_0+s;N)-X^{N-k}_j(NT_0+s;N)}\\
+\sum_{j=1}^{N-k} \frac{\beta/2-1}{X^{N-k+1}_{N-k+1}(NT_0+s;N)-X^{N-k}_j(NT_0+s;N)}
\\-\sum_{j=1}^{N-k} \frac{\beta/2-1}{X^{N-k+1}_{N-k+1}(NT_0+s;N)-X^{N-k+1}_j(NT_0+s;N)} \Biggr)\,\mathrm{d}s \\
- \int_0^t \sum_{i,j=1}^k a_{i,j}\,\frac{\partial^2 f}{\partial x_i\partial
x_j}(\R^N(s))\,\mathrm{d}s,\;\;t\ge0.
\end{split}
\end{equation}
In view of the SDEs \eqref{eq_prelimit_SDE}, \eqref{eq_prelimit_SDE_2} and Ito's formula the process $M^f_N$ is a martingale in the filtration generated by the processes $\big\{X^j_i(NT_0+t),\;t\ge0:\;i=1,\,2,\,\ldots,\,j,\;\;j=N-k,\dots,N\big\}$. Therefore, it is also a martingale in the natural filtration of the process $\R^N$.

\medskip

Our next aim is to prove that for each $T>0$ the following convergences hold:
\begin{eqnarray}
&& \lim_{N\to\infty} \sup_{0\le t\le T} |M^f_N(t)-M^f(t)|=0 \;\;\text{ in probability,\; and} \label{eq_conv_in_prob} \\
&& \forall\,t\in[0,T]:\quad \lim_{N\to\infty} \ev\big[|M^f_N(t)-M^f(t)|\big]=0. \label{eq_conv_in_L1}
\end{eqnarray}

%

To this end, we will now analyze the $N\to\infty$ behavior of the terms in \eqref{eq_martingale_approximation} and show that they converge in the sense of \eqref{eq_conv_in_prob}, \eqref{eq_conv_in_L1} to the corresponding terms in \eqref{eq_martingale}:

\begin{itemize}
\item The convergence $f(\R^N(t))-f(\R^N(0))\to f(R(t))-f(R(0))$, $N\to\infty$ uniformly in $t\in[0,T]$ in probability holds due to \eqref{eq_convergence_of_components} and the continuity of $f$. Since $f$ is bounded, $f(\R^N(t))-f(\R^N(0))$ also converges to $f(R(t))-f(R(0))$ in $L^1$ for every fixed $t\in[0,T]$ by the Dominated Convergence Theorem. \\
\item Similarly, \eqref{eq_convergence_of_components}, the boundedness and continuity of the derivatives of $f$ and $\mathrm{supp}\,f\subset S_\delta$ imply
\begin{eqnarray*}
&& \int_0^t \; \sum_{i=1}^k \; \frac{\partial f}{\partial x_i}(\R^N(s))\,\frac{\beta/2-1}{\R^N_i(s)}\,\mathrm{d}s
\longrightarrow  \int_0^t \; \sum_{i=1}^k \; \frac{\partial f}{\partial x_i}(R(s))\,\frac{\beta/2-1}{R_i(s)}\,\mathrm{d}s,\\
&& \int_0^t \; \sum_{i=1}^{k-1} \; \frac{\partial f}{\partial x_i}(\R^N(s))\,\frac{\beta/2-1}{\R^N_{i+1}(s)}\,\mathrm{d}s
\longrightarrow \int_0^t \; \sum_{i=1}^{k-1} \; \frac{\partial f}{\partial x_i}(R(s))\,\frac{\beta/2-1}{R_{i+1}(s)}\,\mathrm{d}s, \\
&& \int_0^t \; \sum_{i,j=1}^k \; a_{i,j}\,\frac{\partial^2 f}{\partial x_i\partial x_j}(\R^N(s))\,\mathrm{d}s
\longrightarrow \int_0^t \; \sum_{i,j=1}^k \; a_{i,j}\,\frac{\partial^2 f}{\partial x_i\partial x_j}(R(s))\,\mathrm{d}s
\end{eqnarray*}
in the limit $N\to\infty$, both uniformly in $t\in[0,T]$ in probability and in $L^1$ for every fixed $t\in[0,T]$. \\
\item For every $i=1,\,2,\,\ldots,\,k$ the interlacing of the particles on levels $(N-i)$ and $(N-i+1)$ shows the inequality
\begin{multline*}
\quad \quad \quad \Biggl|\sum_{j=1}^{N-i-1} \frac{\beta/2-1}{X^{N-i+1}_{N-i+1}(NT_0+s;N)-X^{N-i}_j(NT_0+s;N)}\\
-\sum_{j=1}^{N-i} \frac{\beta/2-1}{X^{N-i+1}_{N-i+1}(NT_0+s;N)-X^{N-i+1}_j(NT_0+s;N)}\Biggr| \\
 \le \frac{\beta/2-1}{X^{N-i+1}_{N-i+1}(NT_0+s;N)-X^{N-i+1}_{N-i}(NT_0+s;N)}\,.
\end{multline*}

Moreover, Lemma \ref{Lemma_inverse_expectation} and the observation \eqref{eq_in_dist} reveal that the expectation of the latter upper bound tends to $0$ in the limit $N\to\infty$, uniformly in $s\in[0,T]$. Therefore, the boundedness of the derivatives of $f$ implies that the corresponding integrals in \eqref{eq_martingale_approximation} converge to $0$ uniformly in $t\in[0,T]$ in probability and in $L^1$ for any fixed $t\in[0,T]$. \\
\item Lastly, Lemma \ref{lemma_sum_inverse} yields the convergence
$$
\sum_{j=1}^{N-k-1} \frac{\beta/2}{X^{N-k}_{N-k}(NT_0+s;N)-X^{N-k}_j(NT_0+s;N)}\to
\sqrt{\frac{\beta}{2T_0}}
$$
in $L^1$ uniformly in $s\in[0,T]$. Combining this with the boundness and continuity of the
derivatives of $f$ and \eqref{eq_convergence_of_components} we conclude that the sixth line of
\eqref{eq_martingale_approximation} converges to
$$
 -\int_0^t -\sqrt{\frac{\beta}{2T_0}}\,\,\frac{\partial f}{\partial x_n}(R(s))\,\mathrm{d}s
$$
in the limit $N\to\infty$ uniformly in $t\in[0,T]$ in probability and in $L^1$ for any fixed $t\in[0,T]$.
\end{itemize}

\smallskip

At this point, we can combine all the above convergences to obtain \eqref{eq_conv_in_prob} and \eqref{eq_conv_in_L1}. It remains to observe that for any bounded continuous functional $G$ and any $0\le s<t\le T$:
\begin{eqnarray*}
&&\Big|\ev\big[\big(M^f(t)-M^f(s)\big)\,G\big(R(u):\,0\le u\le s\big)\big]\Big| \\
&=&\Big|\ev\big[\big(M^f(t)-M^f(s)\big)\,G\big(R(u):\,0\le u\le s\big)\big]
-\ev\big[\big(M^f_N(t)-M^f_N(s)\big)\,G\big(\R^N(u):\,0\le u\le s\big)\big]\Big| \\
&\le& \Big|\ev\big[\big(M^f(t)-M^f(s)\big)\,G\big(R(u):\,0\le u\le s\big)\big]
-\ev\big[\big(M^f(t)-M^f(s)\big)\,G\big(\R^N(u):\,0\le u\le s\big)\big]\Big| \\
&&+\Big|\ev\big[\big(M^f(t)-M^f(s)\big)G\big(\R^N(u):0\le u\le s\big)\big]
-\ev\big[\big(M^f_N(t)-M^f_N(s)\big)G\big(\R^N(u):0\le u\le s\big)\big]\Big| \\
&&\qquad\qquad\qquad\qquad\qquad\qquad\qquad\qquad\qquad\qquad\qquad\qquad\qquad\qquad\qquad\qquad\qquad\qquad\qquad\longrightarrow0
\end{eqnarray*}
along a suitable subsequence thanks to \eqref{eq_conv_in_prob}, the Dominated Convergence Theorem and \eqref{eq_conv_in_L1}. It follows that $M^f$ is a continuous martingale in the filtration generated by the process $R$ and, hence, also in its own natural filtration.
\end{proof}

Next, for $\delta>0$ we recall \eqref{eq_stop_time_def} and define
$\tau_\delta:=\tau_\delta[\min(R,1/R)]$ where the functions $\min$ and $r\mapsto 1/r$ are applied
componentwise. Now, noting that for every function $f$ of one of the two forms $x\mapsto x_i$ or
$x\mapsto x_i\,x_j$ there is a function $f_\delta\in\F_{\delta/2}$ which coincides with $f$ on
$S_\delta$, using Lemma \ref{Lemma_limit_martingale} for $f_\delta$ and applying the Optional
Stopping Theorem (see e.g. \cite[Section 1.3.C]{KS}) we end up with the following.

\begin{proposition} \label{Proposition_martingales_stopped}
For any $\delta>0$ and any function $f$ of one of the two forms $x\mapsto x_i$ or $x\mapsto x_i\,x_j$ the process $M^f(t\wedge \tau_\delta)$, $t\in[0,T]$, defined via \eqref{eq_martingale}, is a continuous martingale.
\end{proposition}


We now claim that, after extending the underlying probability space if necessary, we can find Brownian motions $W^\delta_1,\,W^\delta_2,\,\ldots,\,W^\delta_k$ such that for all $t\ge 0$
\begin{equation} \label{eq_integral_for_diffs}
\begin{split}
& R_i(t\wedge \tau_\delta)=R_i(0)+ \int_{0}^{t\wedge \tau_\delta}
\bigg(\frac{\beta/2-1}{R_i(s)}-\frac{\beta/2-1}{R_{i+1}(s)}\bigg)\,\mathrm{d}s+
W^\delta_i(t\wedge \tau_\delta),\quad i=1,2,\ldots,k-1,\\
& R_k(t\wedge\tau_\delta)=R_k(0)+\int_0^{t\wedge \tau_\delta}
\bigg(\frac{\beta/2-1}{R_k(s)}-\sqrt{\frac{\beta}{2 T_0}}\bigg)\,\mathrm{d}s+W^\delta_k(t\wedge
\tau_\delta),
\end{split}
\end{equation}
and such that the covariance structure of $W^\delta_1,\,W^\delta_2,\,\ldots,\,W^\delta_k$ is given by
$$
\ev\big[W^\delta_i(t_1)\,W^\delta_j(t_2)\big]=a_{ij}\,t_1\wedge t_2,\quad t_1,\,t_2 \ge 0.
$$
Indeed, one can show the existence of such Brownian motions by repeating the arguments in the proof of \cite[Proposition 5.4.6]{KS}, replacing all occurences of $t$ there by $t\wedge \tau_\delta$ there and using Proposition \ref{Proposition_martingales_stopped}.

\medskip


Finally, note that the joint distribution of the Brownian motions
$W^\delta_1,\,W^\delta_2,\,\ldots,\,W^\delta_k$ is the same as the joint distribution of the
differences $B_1-B_2,\,B_2-B_3,\,\ldots,\,B_k-B_{k+1}$ of i.i.d.\ one-dimensional standard
Brownian motions $B_1,\,B_2,\,\ldots,\,B_{k+1}$. This and the stochastic integral equation
\eqref{eq_integral_for_diffs} allow us to identify the law of $R(t\wedge\tau_\delta)$, $t\in[0,T]$
with that of the weak solution of \eqref{eq_SDE_for_diffs} stopped at $\tau_\delta$. Since
$\delta$ was arbitrary and $\lim_{\delta\to0}\,\tau_\delta=\infty$ with probability one,
Proposition \ref{Proposition_limit_SDE} now readily follows.

\section{Appendix: Properties of the limiting diffusion} \label{Section_Limit_properties}

The limiting object, that is the solution of \eqref{eq_SDE_for_diffs} started from i.i.d.\ Gamma distributions of Theorem \ref{Theorem_convergence_fixed_time}, has several curious properties which we present in this section. The dimension $k$ will vary, so we restore it in the notation and write $R^{(k)}$ for the $(0,\infty)^k$-valued weak solution of
\begin{equation} \label{eq_SDE_for_diffs_k}
\begin{split}
& \mathrm{d}R_i^{(k)}(t)=\frac{(\beta/2-1)\,\mathrm{d}t}{R_i^{(k)}(t)}\,
\mathrm{d}t-\frac{(\beta/2-1)\,\mathrm{d}t}{R_{i+1}^{(k)}(t)}+\mathrm{d}B_i(t)-\mathrm{d}B_{i+1}(t),\quad i=1,2,\ldots,k-1,\\
& \mathrm{d}R^{(k)}_k(t)=\frac{(\beta/2-1)\,\mathrm{d}t}{R^{(k)}_k(t)}
-\sqrt{\frac{\beta}{2T_0}}\,\mathrm{d}t+\mathrm{d}B_k(t)-\mathrm{d}B_{k+1}(t)
\end{split}
\end{equation}
where $B_1,\,B_2,\,\dots,\,B_{k+1}$ are i.i.d.\ one-dimensional standard Brownian motions. Note
that we do not impose a specific initial condition here.

\begin{proposition} \label{proposition_invariant}
For every fixed $k$ the product measure on $(0,\infty)^k$ with one-dimensional marginals of probability density
$$
\frac{1}{\Gamma(\beta/2)}\Big(\frac{\beta}{2T_0}\Big)^{\beta/4}\,x^{\beta/2-1}\,e^{-\sqrt{\frac{\beta}{2T_0}}x}
$$
is invariant for the solution of SDE \eqref{eq_SDE_for_diffs_k}. In other words, the solution of \eqref{eq_SDE_for_diffs_k} started from this measure is a stationary process.
\end{proposition}

\begin{proposition} \label{proposition_restriction}
Suppose that $R^{(K)}$ is started from the product Gamma distribution of Proposition \ref{proposition_invariant}. Then for any $1\le \ell \le m \le K$ the process $\big(R^{(K)}_\ell(t),\,R^{(K)}_{\ell+1},\,\ldots,\,R^{(K)}_m(t)\big)$, $t\ge 0$ is a weak solution of \eqref{eq_SDE_for_diffs_k} with $k=m-\ell+1$.
\end{proposition}

Both Proposition \ref{proposition_invariant} and \ref{proposition_restriction} are immediate
corollaries of Theorems \ref{Theorem_convergence_fixed_time} and
\ref{Theorem_multitime_convergence}. In this section we explain how these facts can be proved
independently (that is, without using the multilevel Dyson Brownian motions of \cite{GS}) if one
can establish the following conjecture.

\begin{conjecture} \label{Feller_conj}
The Markov semigroup of the process $R^{(k)}$ can be extended to a Feller Markov semigroup acting
on the space of continuous functions on $[0,\infty)^k$ vanishing at infinity, endowed with the
uniform norm.
\end{conjecture}

The conjecture seems to be true if one views the process $R^{(k)}$ as a multidimensional
generalization of a one-dimensional Bessel process whose semigroup is Feller on the space of
continuous functions on $[0,\infty)$ vanishing at infinity, endowed with the uniform norm (see e.g.
\cite[Chapter XI, p. 446]{RY}). However, a rigorous proof of the conjecture would require, in
particular, a construction of the process $R^{(k)}$ for initial values on the boundary of
$[0,\infty)^k$ which appears to be beyond reach at the moment.

\medskip

We now proceed to explain how Conjecture \ref{Feller_conj} can be used to prove Propositions
\ref{proposition_invariant} and \ref{proposition_restriction}. To simplify the notation it will be
convenient to choose $T_0=\beta/2$ here (note that this is different from the choice of $T_0$ in
Section \ref{Section_fixed}). For other values of $T_0$ one can either repeat the same arguments or
simply rescale the space and time coordinates as needed.

\medskip

We start with Proposition \ref{proposition_invariant} and claim first that for all infinitely differentiable functions with compact support in $(0,\infty)^k$ the limit
\eq\label{gen_lim}
\lim_{t\to0} \,\frac{\ev\big[f(R^{(k)}(t))|R^{(k)}(0)=x\big]-f(x)}{t}=:(\mathcal{A}^{(k)} f)(x)
\en
exists uniformly in $x\in[0,\infty)^k$ and is given by
\begin{equation}
\label{eq_generator}
\sum_{i=1}^k \frac{\partial^2 f}{\partial x_i^2}(x) - \sum_{i=1}^{k-1}
  \frac{\partial^2 f}{\partial x_i \partial x_{i+1}}(x) + \sum_{i=1}^{k-1}
  \left(\frac{\beta/2-1}{x_i}-\frac{\beta/2-1}{x_{i+1}}\right)\frac{\partial f}{\partial x_i}(x) +
\left( \frac{\beta/2-1}{x_k}-1 \right) \frac{\partial f}{\partial x_k}(x)
\end{equation}
for all $x\in[0,\infty)^k$. This is a consequence of \cite[Theorem 18.11]{Ka}, since one can replace $R^{(k)}$ in the expectation of \eqref{gen_lim} by a diffusion process with bounded coefficients coinciding with those of $R^{(k)}$ on a neighborhood of the support of $f$ and the resulting error is exponentially small in $t^{-1}$ uniformly in $x$. Indeed, the latter error can be bounded above by the product of the supremum of $|f|$ with the probability that $R^{(k)}$ moves by more than a constant amount within the neighborhood of the support of $f$ in time $t$.

\medskip

From the above we conclude that infinitely differentiable functions with compact support in $(0,\infty)^k$ belong to the domain of the generator of $R^{(k)}$. Moreover, a measure $\nu$ on $[0,\infty)^k$ is invariant for the process $R^{(k)}$ if and only if the expectation of $f(R^{(k)}(t))$ does not depend on $t$ for all such functions, when $R^{(k)}$ is started according
to $\nu$. This is in turn equivalent to (see e.g. \cite[Section 4.9]{EK})
\eq
\int_{[0,\infty)^k} (\mathcal A^{(k)}f)(x)\,\nu(\mathrm{d}x)=0
\en
for all functions $f$ as described. Plugging in the candidate invariant measure from the statement of the proposition and integrating by parts we see that it suffices to show
\begin{equation}\label{eq_invariance}
\big(\mathcal{A}^{(k)}\big)^*\prod_{i=1}^k x_i^{\beta/2-1}\,e^{-x_i}=0,
\end{equation}
where
\eq\label{adjoint_op}
\begin{split}
\big(\mathcal{A}^{(k)}\big)^*:=\sum_{i=1}^k \frac{\partial^2}{\partial x_i^2}
- \sum_{i=1}^{k-1} \frac{\partial^2}{\partial x_i \partial x_{i+1}}
- \sum_{i=1}^{k-1} \left(\frac{\beta/2-1}{x_i}-\frac{\beta/2-1}{x_{i+1}}\right)\frac{\partial}{\partial x_i}
- \left( \frac{\beta/2-1}{x_k}-1 \right)\frac{\partial}{\partial x_k} \\
+ \sum_{i=1}^{k}\frac{\beta/2-1}{x_i^2}
\end{split}
\en
is the adjoint of $\mathcal{A}^{(k)}$. To establish \eqref{eq_invariance} we argue by induction over $k$. For $k=1$ one has
\[
\begin{split}
&\big(\mathcal{A}^{(1)}\big)^*\big(x_1^{\beta/2-1}\,e^{-x_1}\big)
=\left(\frac{\partial^2}{\partial x_1^2}-\left( \frac{\beta/2-1}{x_1}-1\right)\frac{\partial}{\partial x_1}
+\frac{\beta/2-1}{x_1^2} \right)\big(x_1^{\beta/2-1}e^{-x_1}\big) \\
&=x_1^{\beta/2-1}e^{-x_1}\left(\frac{(\beta/2-1)(\beta/2-2)}{x_1^2}+1-\frac{\beta-2}{x_1}+\frac{\beta/2-1}{x_1}-1
-\frac{(\beta/2-1)^2}{x_1^2}+\frac{\beta/2-1}{x_1}+\frac{\beta/2-1}{x_1^2} \right)\\
&=0.
\end{split}
\]
Writing $\big(\widehat {\mathcal{A}}^{(k)}\big)^*$ for the same operator as $\big(\mathcal{A}^{(k)}\big)^*$, but acting on
the variables $x_2,\,x_3,\,\ldots,\,x_{k+1}$, we carry out the induction step as follows:
\begin{eqnarray*}
\big(\mathcal{A}^{(k+1)}\big)^*\prod_{i=1}^{k+1} x_i^{\beta/2-1}e^{-x_i}
&=&\big(\widehat{\mathcal{A}}^{(k)}\big)^*\prod_{i=1}^{k+1} x_i^{\beta/2-1}e^{-x_i}
+\big(\mathcal{A}^{(1)}\big)^*\prod_{i=1}^{k+1} x_i^{\beta/2-1}e^{-x_i} \\
&&+ \left( -
  \frac{\partial^2}{\partial x_1 \partial x_{2}} -
  \left(1-\frac{\beta/2-1}{x_{2}}\right)\frac{\partial }{\partial x_1}
 \right)\prod_{i=1}^{k+1} x_i^{\beta/2-1}e^{-x_i} \\
&=&-\frac{\partial}{\partial x_1} \left(\frac{\partial}{\partial x_2}+1-\frac{\beta/2-1}{x_2}\right)
\prod_{i=1}^{k+1} x_i^{\beta/2-1}e^{-x_i} =0. \qquad\qquad\qquad\quad\qedhere
\end{eqnarray*}

\smallskip

We now turn to Proposition \ref{proposition_restriction}. First, we remark that Conjecture \ref{Feller_conj} implies the following uniqueness result for the evolution equation associated with the semigroup of $R^{(k)}$.

\begin{conjecture}\label{lemma_PDE_uniq}
For every infinitely differentiable function $f$ with compact support in $(0,\infty)^k$ the abstract Cauchy problem
\eq\label{ACP}
\frac{\mathrm{d}}{\mathrm{d}t}\,u(t,\cdot)={\mathcal A}^{(k)}u(t,\cdot),\quad u(0,\cdot)=f,
\en
(with ${\mathcal A}^{(k)}$ given by \eqref{eq_generator}) has a unique classical solution in the sense of \cite[Definition II.6.1]{EN}: the mapping $t\mapsto u(t,\cdot)$ from $[0,\infty)$ into the space of continuous functions on $[0,\infty)^k$ vanishing at infinity (endowed with the uniform norm) is continuously differentiable, $u(t,\cdot)$ is in the domain of ${\mathcal A}^{(k)}$ for all $t\ge0$, and \eqref{ACP} holds. Moreover, the solution is given by the application of the semigroup generated by ${\mathcal A}^{(k)}$ to the function $f$.
\end{conjecture}

Indeed, given Conjecture \ref{Feller_conj} one can rely on \cite[Theorem 17.6]{Ka} to conclude that the operator ${\mathcal A}^{(k)}$ generates a strongly continuous semigroup on the space of continuous functions on $[0,\infty)^k$ vanishing at infinity. Conjecture \ref{lemma_PDE_uniq} then becomes a corollary of \cite[Proposition II.6.2]{EN}.

\medskip

We now return to Proposition \ref{proposition_restriction}. From the fact that $R^{(K)}$ satisfies \eqref{eq_SDE_for_diffs_k} with $k=K$ it is clear that $\big(R^{(K)}_l,\,R^{(K)}_{l+1},\,\ldots,\,R^{(K)}_K\big)$ is a weak solution of \eqref{eq_SDE_for_diffs_k} with $k=K-l+1$. Hence, it suffices to prove that the restriction of $\big(R^{(K)}_l,\,R^{(K)}_{l+1},\,\ldots,\,R^{(K)}_K\big)$ to the first $(m-l+1)$ coordinates is a weak solution of \eqref{eq_SDE_for_diffs_k} with $k=m-l+1$. Since one can remove one coordinate process at a time, the proposition boils down to the statement
\eq\label{consistency2}
\big(R^{(n+1)}_1,R^{(n+1)}_2,\ldots,R^{(n+1)}_n\big)\stackrel{d}{=}\big(R^{(n)}_1,R^{(n)}_2,\ldots,R^{(n)}_n\big),\quad
n\in\nn.
\en

\smallskip

To establish \eqref{consistency2} we fix an $n\in\nn$ and an infinitely differentiable function $f$ with compact support in $(0,\infty)^n$, write $R^{(n+1)}=\big(\widetilde{R},R^{(n+1)}_{n+1}\big)$, and define the function
\eq
u(t,x)=\ev\big[f(\widetilde{R}(t))\big|\widetilde{R}(0)=x\big],\quad (t,x)\in[0,\infty)\times(0,\infty)^n.
\en
We note that
\eq\label{whatisu}
u(t,x)=\int_0^\infty \ev\big[f(\widetilde{R}(t))\big|\widetilde{R}(0)=x,\,R^{(n+1)}_{n+1}(0)=y\big]\,\frac{1}{\Gamma(\beta/2)}\,y^{\beta/2-1}\,e^{-y}\,\mathrm{d}y.
\en
Next, we write $v(t,x,y)$ for the conditional expectation in \eqref{whatisu} and apply \cite[Theorem 17.6]{Ka} to the semigroup of the process $R^{(n+1)}$ to obtain the backward Kolmogorov equation
\eq
\frac{\mathrm{d}}{\mathrm{d}t}\,v(t,\cdot,\cdot)=\mathcal{A}^{(n+1)}_{(x,y)}v(t,\cdot,\cdot)
\en
where $\mathcal{A}^{(n+1)}_{(x,y)}$ is the operator of \eqref{eq_generator}, written in the coordinates $(x,y)$ and the equation should be interpreted in the sense of Conjecture \ref{lemma_PDE_uniq}. This allows us to compute
\begin{eqnarray*}
&&\frac{\mathrm{d}}{\mathrm{d}t}\,u(t,\cdot) = \int_0^\infty \Big(\frac{\mathrm{d}}{\mathrm{d}t}\,v(t,\cdot,\cdot)\Big)
\frac{1}{\Gamma(\beta/2)}\,y^{\beta/2-1}\,e^{-y}\,\mathrm{d}y \\
&&= \int_0^\infty \bigg(\mathcal{A}^{(n)}_x\,v(t,\cdot,\cdot)+\mathcal{A}^{(1)}_y\,v(t,\cdot,\cdot)
+\Big(1-\frac{\beta/2-1}{y}\Big)\,\frac{\partial v(t,\cdot,\cdot)}{\partial x_n}
-\frac{\partial^2 v(t,\cdot,\cdot)}{\partial x_n\partial y}\bigg)
\frac{y^{\beta/2-1}\,e^{-y}}{\Gamma(\beta/2)}\,\mathrm{d}y \\
&&= \int_0^\infty \big(\mathcal{A}^{(n)}_x\,v(t,\cdot,\cdot)\big)
\frac{1}{\Gamma(\beta/2)}\,y^{\beta/2-1}\,e^{-y}\,\mathrm{d}y
=\mathcal{A}^{(n)}_x\,u(t,\cdot).
\end{eqnarray*}
In the third identity we have used integration by parts in the $y$ variable, and the identities
\[
(\mathcal{A}^{(1)}_y)^*\,\Big(y^{\beta/2-1}\,e^{-y}\Big)=0,\quad\mathrm{and}
\quad\frac{\partial}{\partial y}\,\Big(y^{\beta/2-1}\,e^{-y}\Big)
=\Big(\frac{\beta/2-1}{y}-1\Big)\,\big(y^{\beta/2-1}\,e^{-y}\big).
\]
We also remark at this point that no boundary terms appear in that integration by parts thanks to the assumption $\beta\ge4$. At this point, Conjecture \ref{lemma_PDE_uniq} implies
\eq\label{basiceq}
u(t,x)=\ev\big[f(\widetilde{R}(t))\big|\widetilde{R}(0)=x\big]=\ev\big[f(R^{(n)}(t))\big|R^{(n)}(0)=x\big],
\quad (t,x)\in[0,\infty)\times(0,\infty)^n,
\en
Indeed, we have just established that the first conditional expectation is a classical solution of the abstract Cauchy problem in Conjecture \ref{lemma_PDE_uniq}. For the second conditional expectation the same follows from another application of \cite[Theorem 17.6]{Ka}.

\medskip

To obtain \eqref{consistency2} we need to improve \eqref{basiceq} to the following statement (which is the Markov property of $\widetilde{R}$): for any $0\le s<t$, $p\in\nn$ and $0\le t_1<t_2<\cdots<t_p\le s$:
\eq\label{maineq}
\ev\big[f(\widetilde{R}(t))\big|\widetilde{R}(t_1),\widetilde{R}(t_2),\ldots,\widetilde{R}(t_p)\big]=u(t-t_p,\widetilde{R}(t_p)).
\en
Indeed, then we can infer that the conditional distribution of $\widetilde{R}(t)$ given $\widetilde{R}(r)$, $0\le r\le s$ is the same as the conditional distribution of $R^{(n)}(t)$ given $R^{(n)}(r)$, $0\le r\le s$. This is due to the fact that the $\sigma$-algebra on the space of continuous functions $[0,s]\to[0,\infty)^n$ is generated by the finite-dimensional projections of such functions.

\medskip

We show \eqref{maineq} by induction over $p$. For $p=1$, \eqref{maineq} is a direct consequence of \eqref{basiceq} and the time-homogeneity of the processes $R^{(n+1)}$. Moreover, for any $p\ge2$, we can appeal to the tower property of conditional expectation and to the Markov property of $R^{(n+1)}$ to rewrite the left-hand side of \eqref{maineq} as
\begin{eqnarray*}
&&
\ev\big[\ev\big[f(\widetilde{R}(t))\big|R^{(n+1)}(t_1),R^{(n+1)}(t_2),\ldots,R^{(n+1)}(t_p)\big]
\big|\widetilde{R}(t_1),\widetilde{R}(t_2),\ldots,\widetilde{R}(t_p)\big] \\
&&=\ev\big[\ev\big[f(\widetilde{R}(t))\big|R^{(n+1)}(t_p)\big]\big|\widetilde{R}(t_1),\widetilde{R}(t_2),\ldots,\widetilde{R}(t_p)\big] \\
&&=\ev\big[v(t-t_p,\widetilde{R}(t_p),R^{(n+1)}_{n+1}(t_p))\big|\widetilde{R}(t_1),\widetilde{R}(t_2),\ldots,\widetilde{R}(t_p)\big]
\end{eqnarray*}
where $v$ stands for the conditional expectation in \eqref{whatisu} as before. It remains to prove that $R^{(n+1)}_{n+1}(t_p)$ is independent of $\big(\widetilde{R}(t_1),\widetilde{R}(t_2),\ldots,\widetilde{R}(t_p)\big)$, since then
\begin{eqnarray*}
\ev\big[v(t-t_p,\widetilde{R}(t_p),R^{(n+1)}_{n+1}(t_p))\big|\widetilde{R}(t_1),\widetilde{R}(t_2),\ldots,\widetilde{R}(t_p)\big]
=\int_0^\infty v(t-t_p,\widetilde{R}(t_p),y)\,\frac{1}{\Gamma(\beta/2)}\,y^{\beta/2-1}\,e^{-y}\,\mathrm{d}y \\
=u(t-t_p,\widetilde{R}(t_p)).
\end{eqnarray*}

\smallskip

To show the independence assertion we argue again by induction over $p\ge2$. For $p=2$ the assertion is equivalent to the statement
\eq\label{indep}
\ev\big[g(\widetilde{R}(t_2))\,h(R^{(n+1)}_{n+1}(t_2))\big|\widetilde{R}(t_1)=r_1\big]
=\ev\big[g(\widetilde{R}(t_2))\big|\widetilde{R}(t_1)=r_1\big] \,\int_0^\infty
h(y)\,\frac{1}{\Gamma(\beta/2)}\,y^{\beta/2-1}\,e^{-y}\,\mathrm{d}y
\en
for all infinitely differentiable functions $g$, $h$ with compact supports in $(0,\infty)^n$, $(0,\infty)$, respectively, and any $r_1\in(0,\infty)^n$. Now, note that the left-hand side of \eqref{indep} can be written as
\begin{eqnarray*}
\int_0^\infty
\ev\big[f(\widetilde{R}(t_2))\,g(R^{(n+1)}_{n+1}(t_2))\big|R^{(n+1)}(t_1)=(r_1,y)\big]\,
\frac{1}{\Gamma(\beta/2)}\,y^{\beta/2-1}\,e^{-y}\,\mathrm{d}y \\
=:\int_0^\infty w(t_2-t_1,r_1,y)\,
\frac{1}{\Gamma(\beta/2)}\,y^{\beta/2-1}\,e^{-y}\,\mathrm{d}y
\end{eqnarray*}
where the hereby defined function $w$ is a classical solution of the abstract Cauchy problem
\eq
\frac{\mathrm{d}}{\mathrm{d}t}\,w(t,\cdot,\cdot)=\mathcal{A}^{(n+1)}_{(r_1,y)}\,w(t,\cdot,\cdot),\qquad w(0,r_1,y)=f(r_1)\,g(y)
\en
in the sense of Conjecture \ref{lemma_PDE_uniq}. Now, the same argument as the one following \eqref{whatisu} reveals that both sides of \eqref{indep} are classical solutions of the abstract Cauchy problem
\eq
\frac{\mathrm{d}}{\mathrm{d}t}\,\widetilde{u}(t,\cdot)=\mathcal{A}^{(n)}_{r_1}\,\widetilde{u}(t,\cdot),
\qquad \widetilde{u}(0,r_1)=g(r_1)\,
\int_0^\infty h(y)\,\frac{1}{\Gamma(\beta/2)}\,y^{\beta/2-1}\,e^{-y}\,\mathrm{d}y
\en
in the sense of Conjecture \ref{lemma_PDE_uniq}. Thus, \eqref{indep} is a consequence of Conjecture \ref{lemma_PDE_uniq}.

\medskip

For the induction step we pick functions $g$, $h$ as before and vectors $r_1,r_2,\ldots,r_{p-1}\in(0,\infty)^n$, and apply the induction hypothesis, the Markov property of $R^{(n+1)}$ and the statement in the $p=2$ case subsequently to compute
\begin{eqnarray*}
&&\ev\big[g(\widetilde{R}(t_p))\,h(R^{(n+1)}_{n+1}(t_p))\big|\widetilde{R}(t_1)=r_1,\ldots,\widetilde{R}(t_{p-1})=r_{p-1}\big] \\
&=&\int_0^\infty
\ev\big[g(\widetilde{R}(t_p))\,h(R^{(n+1)}_{n+1}(t_p))\big|\widetilde{R}(t_1)=r_1,\ldots,
\widetilde{R}(t_{p-1})=r_{p-1},R^{(n+1)}_{n+1}(t_{p-1})=y\big]
\frac{y^{\beta/2-1}}{\Gamma(\beta/2)} e^{-y}\mathrm{d}y \\
&=&\int_0^\infty \ev\big[g(\widetilde{R}(t_p))\,h(R^{(n+1)}_{n+1}(t_p))\big|
\widetilde{R}(t_{p-1})=r_{p-1},R^{(n+1)}_{n+1}(t_{p-1})=y\big]
\,\frac{y^{\beta/2-1}}{\Gamma(\beta/2)}\,e^{-y}\,\mathrm{d}y \\
&=&\ev\big[g(\widetilde{R}(t_p))\big|\widetilde{R}(t_{p-1})=r_{p-1}\big]\,\int_0^\infty h(y)\,
\frac{y^{\beta/2-1}}{\Gamma(\beta/2)}\,e^{-y}\,\mathrm{d}y.
\end{eqnarray*}
Carrying out the same computation backward with $h\equiv1$ we see that the latter expression is
given by
\[
\ev\big[g(\widetilde{R}(t_p))\big|\widetilde{R}(t_1)=r_1,\ldots,\widetilde{R}(t_{p-1})=r_{p-1}\big]\,
\int_0^\infty h(y)\, \frac{y^{\beta/2-1}}{\Gamma(\beta/2)}\,e^{-y}\,\mathrm{d}y.
\]
It follows that $R^{(n+1)}_{n+1}(t_p)$ is independent of $\big(\widetilde{R}(t_1),\,\widetilde{R}(t_2),\,\ldots,\,\widetilde{R}(t_p)\big)$ as desired.

\bigskip\bigskip


\quad\\

\end{document}